\documentclass[12pt,twoside]{amsart}
\usepackage{amssymb}
\usepackage{a4wide}
\usepackage{times}

\def\hpq0{h^{p,q}_{\leq 0}}
\def\Hpq0{\H_{\leq 0}^{p,q}}

\def\dbar{\bar\partial}

\def\R{{\mathbb R}}

\def\C{{\mathbb C}}

\def\H{{\mathcal H}}

\def\Re{{\rm Re\,  }}

\def\be{\begin{equation}}
\def\ee{\end{equation}}

\newtheorem{mainthm}{Theorem}

\newtheorem{lemma}{Lemma}[section]
\newtheorem{remark}{Remark}[section]
\newtheorem{maincor}[mainthm]{Corollary}
\newtheorem{cor}[mainthm]{Corollary}
\newtheorem{prop}[mainthm]{Proposition}

\theoremstyle{definition}

\theoremstyle{remark}

\newtheorem{preex}{Example}

\numberwithin{equation}{section}

\def\opcit{\underbar{\phantom{aaaaa}}}

\def\MA{Monge--Amp\`ere }

\def\dbar{\bar\partial}

\def\ra{\rightarrow}
\def\eps{\epsilon}
\newcommand{\RR}{\mathbb{R}}
\newcommand{\CC}{\mathbb{C}}
\newcommand{\DD}{\mathbb{D}}

\def\beq{\begin{equation}}
\def\eeq{\end{equation}}
\def\bi#1{\bibitem{#1}}
\def\PSH{\mathrm{PSH}}

\def\b#1{\bar{#1}}

\def\bclaim{\begin{claim}}
\def\eclaim{\end{claim}}
\def\bdefin{\begin{defin}}
\def\edefin{\end{defin}}
\def\bcor{\begin{cor}}
\def\ecor{\end{cor}}
\def\bthm{\begin{thm}}
\def\ethm{\end{thm}}
\def\blem{\begin{lem}}
\def\elem{\end{lem}}
\def\blemma{\begin{lem}}
\def\elemma{\end{lem}}
\def\bprop{\begin{prop}}
\def\eprop{\end{prop}}
\def\bremark{\begin{remark}}
\def\eremark{\end{remark}}
\def\bpf{\begin{proof}}
\def\epf{\end{proof}}
\def\beq{\begin{equation}}
\def\eeq{\end{equation}}
\def\beqno{\begin{equation*}}
\def\eeqno{\end{equation*}}
\def\eaeq{\end{aligned}}
\def\baeq{\begin{aligned}}

\def\l{\lambda}
\def\bal{\begin{aligned}}
\def\eal{\end{aligned}}

\def\pD{\partial\DD}
\def\tphi{{\tilde\phi}}

\def \cF {\mathcal{F}}

\def\eps{\varepsilon}

\newcommand{\upchi}{\raise1pt\hbox{$\chi$}}

\def\longto{\longrightarrow}
\def\benu{\begin{enumerate}}
\def\eenu{\end{enumerate}}
\newcommand{\AND}{\qquad {\rm and} \qquad}

\def \CC {\mathbb{C}}
\def \RR {\mathbb{R}}
\def \vphi {\varphi}

\def \Re {\hbox{\rm Re}\,  }

\begin{document}

\title[]
{Complex interpolation of $\R$-norms, duality and foliations}

\author[]{Bo Berndtsson, Dario Cordero-Erausquin, Bo'az Klartag and Yanir A. Rubinstein}

\begin{abstract}
The complex method of interpolation, going back
to Calder\'on and Coifman et al., on the one
hand, and the Alexander--Wermer--Slodkowski
theorem on polynomial hulls with convex fibers,
on the other hand, are generalized
to a method of interpolation of real (finite-dimensional) Banach spaces
and of convex functions.
The underlying duality in this method is given
by the Legendre transform. Our results can also be interpreted as new properties of solutions of the homogeneous complex Monge--Amp\`ere equation.
\end{abstract}

\maketitle

\tableofcontents

\section{Introduction and background on complex interpolation}

Is it possible to perform complex interpolation of real Banach spaces? What are the symmetries of the homogenous complex Monge--Amp\`ere equation?
In the present article, we will
show that these two questions have a partial common answer.

\medskip Arguably, it is a bit unreasonable to ask for complex interpolation of real Banach spaces;  moreover, why would one want to do such a thing? One motivation comes from the geometry of convex bodies. The classical Brunn-Minkowski theory deals with convex combinations $(1-t)K+tL$ of convex bodies, and this is also the operation that
real interpolation relies upon. But for several interesting geometric problems, one would like to consider  geometric means "$K^{1-t} L^t$" of bodies, in particular when it comes to duality (or rather polarity, for bodies). A crucial property of complex interpolation is that it commutes (in an isometric way) with duality of Banach spaces, and therefore, if we interpolate between a space and its (conjugate) dual, we find ``in the middle'' the Euclidean space. This property  can be used to recover some special cases of the Blaschke-Santal\'o inequality, for instance,
see Berndtsson \cite{Be} and Cordero-Erausquin \cite{ce}. Geometric means of bodies appear also in the context of the log-Brunn-Minkowski 
problem of B\"or\"oczky, Lutwak, Yang and Zhang \cite{BLYZ}.

\medskip Here we will not focus on the most important feature of complex interpolation, the ability to interpolate linear operators.
Rather,  {\it we seek for an interpolation procedure that will commute with duality.}
%YR
Before going on, let us recall what complex interpolation is, and let us defer the announcement of our new results to
Section \ref{sec_main}.
The complex method of interpolation, as introduced by
Calder\'on \cite{Calderon0,Calderon} and Lions \cite{Lions},
is originally a way to associate to a given pair of complex Banach spaces, $X_0$ and $X_1$, a family of
intermediate Banach spaces $X_\theta$, $\theta \in [0,1]$. This idea was later generalized by Coifman, Cwikel, Rochberg, Sagher, and Weiss~\cite{CCRSW, CCRSW0} to the following setting
(the situation for domains of dimension larger than two was also studied
in the literature, and is quite involved,
cf. Coifmann--Semmes \cite{CS}, and will not be developed here):
%YR
 for each $s$ on the boundary $\pD$ of the unit disk $\DD:=\{t\in\CC\,:\, |t|<1\}$
% in $\C$
we are given a complex  Banach space $X_s$, and we construct interpolated spaces $X_t$ at any $t\in \DD$. This includes the classical interpolation in the following way. If we are given $\theta\in (0,1)$ and two Banach spaces $X$ and $Y$, then we put $X_s=X$ on an arc of $\pD$ of  length $2\pi(1-\theta)$ and $X_s=Y$ on the complement arc of length $2\pi\theta $; then the space $X_0$ that will be constructed at $0\in \DD$ will be the usual complex interpolated space between $X$ and $Y$ with  parameter $\theta$.

In this article we will
assume that all the spaces are of the same finite dimension, so that they can all be
identified as vector spaces with $\C^n$ (while the infinite-dimensional case entails
considerable additional technical difficulties, the $\CC^n$ case
already contains many of the key ideas). We will also agree that an \emph{$\R$-norm} refers to a norm on the $\R$-vector space $\R^{2n}=\C^n$, whereas a \emph{$\C$-norm}, or simply a norm,  refers to a norm on the $\C$-vector space $\C^n$. Therefore $\|\cdot\|$ is a $\C$-norm if and only if it is an $\R$-norm with the extra property that $\| e^{i\theta}w\|=\|w\|$ for every $w\in \C^n$ and $\theta\in \R$, or equivalently
\begin{equation}\label{Csymmetry}
\| \lambda w\|=|\lambda | \|w\|, \qquad \forall w\in \C^n, \forall \lambda \in \C.
\end{equation}
Back to interpolation, since as a linear space $X_s$  is always $\C^n$, the only thing that varies with the parameter $s\in \partial \DD$ is therefore
the norm, which we denote by $\|\cdot\|_s$. Besides measurability of the family of norms (i.e., $(s, w) \to \|w\|_s$ is measurable on $\pD \times \CC^n$), some weak integrability is usually required. To avoid technicalities let us assume the strong (but not too restrictive) property that  the norms are (uniformly) equivalent, i.e.,  there exists $c,C>0$ such that $c |w| \le \|w\|_s \le C |w|$ for all $w\in \C^n$ and $s\in \pD$, where $|\cdot|$ stands for the Euclidean/Hermitian norm on $\R^{2n}=\C^n$. Complex interpolation then produces,
for each $t$ inside the disk $\DD$, a norm
which we denote in the same way by $\| \cdot \|_t$,
and has the given norms as boundary values. This norm is defined as follows:
  for $w_0\in \C$ and $t_0\in \DD$,
\begin{multline}\label{definterpolation}
\|w_0\|_{t_0} := \inf\Big\{ {\rm ess} \sup_{s\in \partial \DD} \|f(s)\|_s \; : \\  f: \DD \to \C^n \textrm{ is bounded and holomorophic},
\textrm{with $f(t_0)=w_0$}\Big\}.
\end{multline}
Here we will always use the notation  $f(s) = \lim_{r \rightarrow 1^-} f(rs)$ for the radial boundary values of the function $f$ at (almost-all) $s \in \pD$.
In this framework, the duality theorem for complex interpolation takes a very nice form. Here, the dual norm $\|\cdot\|^\star$ to an $\R$ or $\C$ norm  $\|\cdot\|$  refers to the norm defined by
\begin{equation}\label{dualnorm}
\|z\|^\star := \sup_{\|w\|\le 1} \Re(z \cdot w)
\end{equation}
where $z \cdot w = \sum z_j w_j$. 
%DCE
We favour the notation  $\star$ for the dual norm and dual space, not to confuse it with $\ast$ which will stand for the Legendre transform. 
In the case where $\|\cdot\|$ is a $\C$-norm, this coincides with the usual complex dual, i.e.
$$\|z\|^\star = \sup_{\|w\|\le 1} |z \cdot w|.$$
In all cases, our notion of dual corresponds rather to the \emph{conjugate} dual because we use $z\cdot w$ rather than $\overline z \cdot w$. But this notion is better adapted to interpolation, so we will adopt it here, keeping in mind that our dual space is traditionally called the conjugate dual space. The duality theorem for complex interpolation expresses the following remarkable property:  if we take as boundary norms $N_s := \|\cdot\|^\star_s$,  where $\|\cdot\|_s$ are as before $\C$-norms, then we have for every $t\in \DD$,
$N_t = \|\cdot \|_t^\star $. In other words, the interpolated norm of the dual norms, is the dual norm of the interpolated norms	.

\medskip Formally, we can use formula~\eqref{definterpolation} even in the case where the given norms $\|\cdot\|_s$ are only $\R$-norms, but the rest of the article will show that there are better choices (in particular, the duality theorem we just mentioned would fail). Before presenting other choices, let us fix some notation regarding harmonic extension. We will denote by $P[h]$, or with some convenient abuse of notation $P[h(s)]$, the Poisson extension to $\DD$ of a function $h\in L^1(\partial \DD):=L^1(\partial \DD, \sigma)$, where $d\sigma(s) = \frac{ds}{2\pi}$ is the usual probability measure on the circle;   namely
\beq\label{HarmMeasEq}
P[h](t):= P[h(s)](t):=\int_{\pD}P(t,s) \, h(s)\, d\sigma(s) \quad \textrm{ where }  P(t,s) := \frac{1-|t|^2}{|s-t|^2}.
\eeq
The function $P[h]$ is harmonic in $\DD$. When needed, we will also denote by $P[L^r]$ the space of all harmonic functions of the form $P[f]$ where $f \in L^r(\pD, \CC)$, for $r \geq 1$.
It is well-known that the radial boundary values of $P[f]$ exist almost everywhere in $\pD$, and equal to the function $f$ itself
(see, e.g., Katznelson \cite[Section I.3.3]{katz}). Recall also that the Hardy space $ H^r(\DD, \CC) $
is the subspace of all holomorphic functions in $P[L^r]$, where
we identify a function in $P[L^r]$ with its boundary values.  Then, a reasonable definition for ``interpolation'' of $\R$-norms $\|\cdot\|_s$ would be
 \begin{equation}\label{definterpolation2}
\|w_0\|_{t_0} := \inf\Big\{ \sqrt{P\big[\|f(s)\|_s^2 \big](t_0)} \; :   f: \DD \to \C^n \textrm{ with } f\in H^2(\DD, \CC^n)
\textrm{ and } f(t_0)=w_0\Big\}.
\end{equation}
In the present article, $ H^p(\DD, \C^n) = H^p(\DD, \CC) \otimes \CC^n$ is the  usual Hardy space (of index $p$) of holomorphic functions on the disc $\DD$ with values on $\C^n$ (note that the holomorphic functions in our article will have values on $\C^n$).

\medskip In the case of $\C$-norms, the two definitions~\eqref{definterpolation} and~\eqref{definterpolation2} coincide. This crucial fact relies on the following observation, the proof of which will be recalled later in section \ref{sec_last}. Select $t_0\in \DD$ and $w_0\in \C^n$, with $\|w_0\|_{t_0}=1$, say. Then it is possible to find a bounded,
 holomorphic function $f: \DD \to \C^n$ with $f(t_0)=w_0$ such that for almost any $s \in \partial \DD$,
\beq\label{trivialfoliation}
 \|f(s)\|_s = 1.
\eeq
The construction of such $f$ uses, as we shall see,   the property~\eqref{Csymmetry} of $\C$-norms. This provides a ``foliation'' by holomorphic discs for the function $(t,w)\to \|w\|_t$ so that the function is constant along each leaf. It follows that whatever kind of ``average'' we take of the boundary values $\|f(s)\|_s$ along a leaf, we will get the same quantity, and in particular~\eqref{definterpolation} and~\eqref{definterpolation2} coincide.

\medskip
The situation for $\R$-norms is drastically different, as it is not always possible to construct a function $f$ satisfying~\eqref{trivialfoliation}, and as a matter of fact~\eqref{definterpolation} and~\eqref{definterpolation2} would no longer give the same extension.  However, with the choice of~\eqref{definterpolation2}, we will see that the duality principle still holds,
and moreover that a slightly different foliation still exists.
The duality principle has then to be understood in terms of Legendre's transform of convex functions, as we will see.
In fact, we will see that what is meaningful in the case of $\R$-norms, is to consider ``interpolation'' of powers of norms, $\|\cdot\|_s^p$, for $p\in [1, +\infty)$. Unlike the case of $\C$-norms, this leads to different interpolants for different $p$'s.

\medskip
Foliations  connect
complex interpolation to the homogenous complex Monge--Amp\`ere equation (HCMA).
Let us set $\phi(t,w) := \|w\|_t$ where $\phi:\overline \DD\times \C^n \to \R$. We view  $\phi$ as a function of $(n+1)$ variables $(w_0, w_1, \ldots, w_n)$ on  $\overline \Omega$ where
$$\Omega:=\DD\times \C^n\subset \C^{n+1}=\{(w_0,w) \; ; w_0\in \C , w\in \C^n\}$$
with  $\phi(s,w)=\|w\|_s$ on $\partial \Omega= \pD\times \C^n$  given by a family of $\C$-norms, extended to $\Omega$ by~\eqref{definterpolation}. It turns out that $\phi$ is pluri-subharmonic ($\PSH$). This classical fact is not obvious from the definition~\eqref{definterpolation}, but it follows from the interpolation duality theorem recalled above.  Then, maybe in some weak sense, we have the homogenous complex Monge--Amp\`ere equation (HCMA),
\beq\label{MAnorm}
\det \nabla^2_{\CC} \phi = \det\Big(\frac{\partial^2 \phi}{\partial w_j \partial \overline{w_k}}\Big)_{0\le j,k\le n}=0 \qquad \textrm{ on } \Omega.
\eeq
The reason is indeed the existence of the foliation~\eqref{trivialfoliation}. For a fixed $(t_0,w_0)\in\Omega$ and $f$ as in~\eqref{trivialfoliation}, if we introduce the holomorphic function $\alpha:\DD\to \Omega$, $\alpha(z)=(z,f(z))$, then $z\to \phi(\alpha(z))$ is harmonic, since it is constant (!), and so its Laplacian vanishes. So at $t_0$ we have (in some weak sense) that $(\nabla^2_{\CC} \phi ) \frac{\partial \alpha}{\partial z} \cdot \overline{\frac{\partial \alpha}{\partial z}} = 0$, which in turn implies that
the vector $\frac{\partial \alpha}{\partial z} \in \CC^{n+1}$ lies in the kernel of the matrix $\nabla^2_{\CC} \phi$, since the matrix is nonnegative, and hence ~\eqref{MAnorm} holds.

\medskip
Let us discuss an example, the case where we are given multiples of the Hilbert norm $|\cdot|$ on $\C^n$ (take $n=1$, say), that is
$
\|w\|_s= |w| e^{-u(s)},
$ for $w\in \C^n$ and $s\in \pD$
where the real-valued function $u$ is bounded, say, on $\pD$.
Then, for given $t_0\in\DD$, using the holomorphic function $f(t) = w \, e^{H(t)-H(t_0)}$, where $H$ is the
%$\C$-valued
holomorphic function whose boundary values have real part equal to $u$, it is easy to see from definition~\eqref{definterpolation} that the interpolating norms inside the disk are  given by
$$
\|w\|_{t_0}= |w| e^{-P[u](t_0)} = | w e^{-H(t_0)}|.$$
In particular $t \to (t, f(t))$ provides a holomorphic disc  through $(t_0,z)$ with $t\to \|f(t)\|_t$ constant, equal to $\|z\|_{t_0}$.
More generally, if the boundary norms are Hilbert norms on $\CC^n$,
$$
\|z\|^2_s:=\sum a_{jk}(s)z_j\bar z_k,
$$
the interpolating norms are still Hilbert norms. Complex interpolation is therefore a method to extend the matrix valued function $A(s)=[a_{jk}(s)]$ defined on the boundary of the disk to its interior, and in the terminology of  hermitian  metrics on  the  vector bundle $\DD\times \C^n$, this is done so that the curvature of the extension vanishes, i.e.,
\beq \label{matrix}
\Theta^A:= \dbar (A^{-1}\partial A)=0.
\eeq
The vanishing of the curvature means that the hermitian metric defined by $A$ is holomorphically   equivalent to the standard Euclidean metric. In other words, there exists a matrix valued function $M(t)$, holomorphic in $t \in  \DD$, such that $A = M^* M$ throughout the disc. The basic interpolation theorem in this setting is due to Wiener and Masani, see \cite{Wiener-Masani}.

\medskip We proceed with a somewhat dual point of view in interpolation theory. Let us regard the norms as being defined by their unit balls
$$
B_s=\{z\in \C^n; \|z\|_s\leq 1\}.
$$
The definition (\ref{definterpolation}) of the complex norm $\|\cdot\|_t$ is equivalent
to the following definition in terms of the unit balls $B_t$ of the norm: For any $t \in \DD$,
\be\label{BtEq}
B_t =\{ h(t)\; ; \  h\in H^{\infty}(\DD, \C^n),  h(s)\in B_{s}
\quad\text{for almost all} \, s\in \pD\}.
\ee
By using a polar body (or dual norms) and the duality theorem,  we can view this unit ball $B_t$ as some {\it affine hull}: For $t \in \DD$,
$$B_t:=\{ w \in \C^n \; ;   \;   \forall f\in H^{\infty}(\DD, \C^n),   \; \Re(w \cdot f(t) ) \le \sup_{s\in \partial \DD, \|z\|_s \le 1}   \Re(z \cdot f(s) ) \}.$$
In fact, at least in the strictly-convex
case, a foliation for the norms $\| \cdot \|_t$ induces a foliation for the dual norms $\| \cdot \|_t^*$,
and constructing foliations for such hulls is closely related to the duality theorem for interpolation (this follows from the work of \cite{CCRSW} but also from
Corollaries 4 and 6 below).
%YR

\medskip
An interesting result regarding foliation for hulls, that requires only convexity and not $\C$-symmetry, is the Alexander--Wermer--Slodkowski theorem (AWS  theorem) \cite{AW,slodki86}
which we now describe.
Given a compact set $K\subset \pD\times\CC^n$
whose fibers $K_s=\{ z \in \CC^n \, : \, (s, z)\in K\}$ are {\it convex sets} in $\CC^n$, the AWS theorem stipulates that $\hat K$,
the polynomial hull of $K$,
\beq\label{KPEq}
\hat K:=\{(t,z)\in \bar\DD\times\C^n\,:\, |P(t,z)|\leq\sup_K |P|, \,
\text{for all holomorphic polynomials}\, P\},
\eeq
consists of all analytic discs whose boundary (in an a.e. sense) lies in $K$. In other words, for any $(t_0, z_0)\in \hat K$, there exists a bounded holomorphic function $f:\DD\to \C^n$ such that $f(t_0)=z_0$ and
$f(s)\in K_s$  for a.e. $s\in \pD$ (so the analytic disc through $(t_0, z_0)$ is   given by $t\to (t,f(t))$). Note that such $f$ satisfies $f(t) \in \hat{K}_t =\{(t, \cdot) \in \hat K\}$, since for every polynomial $P$, the function $t\to |P(t,f(t))|$ is subharmonic.

\medskip
Connections between the AWS theorem and classical complex interpolation  have been already put forward by Slodkowski in~\cite{slodki86, slodki}. In particular, in~\cite{slodki} Slodkowski studied duality and was able to reproduce the duality theorem in the case of $\C$-norms. But despite some partial results, the picture  remained incomplete when it came to duality for $\R$-norms. Since the AWS theorem assumed only convexity of the fibers, further relations should be revealed between interpolation and duality even if we work with $\R$-norms. To this end, we start with three new observations:
First, that in the AWS theorem, we can use only polynomials that are {\it linear} in $z$, namely polynomials of the form $(t,z) \to z\cdot \mathcal P(t) + Q(t)$ where $\mathcal P$ and $Q$ are polynomials from $\C$ to $\C^n$ and $\C$, respectively. Second, when we work with $\R$-norms, the passage from norms or power of norms to bodies is less convenient as explained above, in particular the duality theorem needs to be modified. The right concept for duality will be the Legendre transform. Here, the Legendre transform of a convex function $\varphi: \C^n \to \R\cup\{+\infty\}$ is defined by
$$\varphi^\ast(z) := \sup_{w\in \C^n}\big\{ \Re(w\cdot z) - \varphi(w)\big\}.$$
With respect to the usual definition, there is a conjugation missing, so we are really working with the conjugate Legendre transform. The third observation is that
the homogeneity of the norm plays a minor r\^ole, and that the class of convex functions, which are not necessarily $\RR$-homogenous, is suitable  in our context.

\medskip Our conjugate Legendre transform is order-reversing, i.e., if $\vphi_1 \leq \vphi_2$ then $\vphi_2^* \leq \vphi_1^*$. Since $\vphi(z) + \vphi^*(\bar{z}) \geq |z|^2$ for all $z$,
the unique function $\vphi$ satisfying $\vphi^*(\bar{z}) = \vphi(z)$ is the function $ \vphi(z) = |z|^2 / 2$ on $\CC^n$.
From the point of view of the Monge--Amp\`ere equation,
our duality theorem below states that a weak solution $\phi(t,z) \ (t \in \DD, z \in \CC^n)$ of the HCMA
equation (\ref{MAnorm}) which is convex in $z$ is transformed to another such solution by an application of the Legendre transforms in the $z$-variables.
In a separate article, we  emphasize this point of view, of the Legendre transform as a ``symmetry'' of the HCMA
equation around the fixed point $$ I(z) = |z|^2/2 \qquad \qquad \qquad (z \in \CC^n). $$
%
%\medskip 
%YR
Moreover, in that article we  show that a real analytic PSH function $\omega$ defines a local
$\omega$-Legendre transform, essentially by replacing the expression $\Re w\cdot z$ by the so-called {\it Calabi diastasis function} defined by $\omega$.
The $\omega$-Legendre transform is defined on an open neighbourhood of $\omega$ and it has $\omega$ as a fixed point. Moreover, by using the work of Lempert  \cite{Lempert} we can show that the $\omega$-Legendre transform is again a ``symmetry'' of the HCMA equation,
since it transforms a solution of the HCMA equation to another solution of the same equation.
These additional symmetries render the space of PSH functions, or rather, the space of K\"ahler metrics, a locally-symmetric space
with respect to the Mabuchi metric. Explanations and details in \cite{never}.

\medskip It is now time to move to the new results of this article, which we present in the next section.
Subsequent sections are devoted mostly to the proof of these results.
It is a pleasure to acknowledge that in addition to the Alexander-Wermer-Slodkowski theorem, our point of view in this article is greatly influenced by the works of Rochberg \cite{Rochberg}
and Semmes \cite{Semmes1988, S}.
Moreover, we thank the referees for their careful reading
and many insightful comments. In particular, we learned from them
about the somewhat morally related work of Royden--Wong \cite{RW}
who used dual extremal problems in complex analysis 
to study the Kobayashi metric in convex domains,
(and by Slodkowski \cite[\S 2]{slodki} the problem
of complex geodesics they studied can be related to 
an interpolation problem of the AWS type). 
The work of Poletsky (\cite{Poletsky}) on disk functionals is also related to our work, in particular to the identity between  the functions $\check\phi$ and $\hat\phi$ defined in the next section. The difference is that whereas Poletsky considers Poisson integrals of functions $\phi\circ f$ where $\phi$ is defined in the whole domain (in our case $\DD\times\C^n$)  and $f$ is an arbitrary analytic disk in the domain, we consider only disks that are graphs, with boundary in the boundary of the domain, and $\phi$  defined only on the boundary.

%%%%%%%%%%%%%%%%%%%%%%%%%%%%%%%%%%%%%%%%%%%%%%%%%%%%%%%%%%%%%%%%%%%%%%%%%%%%%%%%%%%%%%%%%%%%%%%%%%%%%%

\section{Main results}
\label{sec_main}

Our framework allows us to encompass boundary data more general  than power of norms. We assume that we are given a Borel measurable function $\phi:\pD\times \C^n\to \R \cup \{ + \infty \}$
which is {\it fiberwise convex}, i.e., for every $s\in \pD$,  the function
$$\phi_s(\cdot) := \phi(s, \cdot ):\C^n\to \R \cup \{ + \infty \} $$
 is \emph{convex} on $\C^n$. We allow $\phi$ to attain the value $+\infty$, though it is usually not needed.
 The convexity of $\phi_s$ means that the set $\{ z \in \CC^n \, ; \, \phi_s(z) < \infty \}$ is convex, $\phi_s$ is a finite, convex
 function on this set.  For simplicity, in this article we shall make the technical assumption of uniform growth.

For $p \in (1, +\infty]$ we say that the boundary data $\phi:\pD\times \C^n\to \R$  satisfies {\it $p$-uniform growth conditions}, if for some $c, C, A > 0$,
 \beq\label{condnorms}
 \forall s \in \pD, \forall w\in \C^n, \qquad   \left \{ \begin{array}{llll} c |w|^p - A \le  \phi_s(w)  \le C |w|^p + A & & \text{when}& p \in (1, +\infty), \\
\chi(c |w|) - A \leq  \phi_s(w)  \leq \chi(C |w|) + A & &\text{when} & p = +\infty,
 \end{array} \right.  \eeq
where $\chi(t) = 0$ for $t \leq 1$ and $\chi(t) = +\infty$ for $t > 1$. In the case where $p = +\infty$, we also require that $\phi_s$ is lower-semi continuous
for any $s \in \pD$.
The case $p = +\infty$ appears less exceptional from the point of view of the {\it fiberwise Legendre transform}.
For fixed $s$, we shall agree to use the following equivalent notations for the fiberwise Legendre transform:
$$\phi^\ast(s, z):=(\phi^\ast)_s(z) := \phi^\ast_s(z)=(\phi_s)^\ast(z)= \sup_{w \in \C^n} \big\{\Re(z\cdot w) - \phi(s, w)\big\}. $$
We emphasize that throughout the article the duality is only over $\CC^n$, with the parameter $s\in \pD$ or later $t\in \DD$ being fixed. By standard convex analysis (see, e.g., \cite[Theorem 12.2]{Rockafellar}),
whenever $\phi$ is a fiberwise convex function satisfying $p$-uniform growth conditions,
$$ ( \phi^{\ast} )^{\ast} = \phi. $$
Denoting $q = p / (p-1) \in [1, \infty)$, we may express the uniform
growth condition (\ref{condnorms}) equivalently as follows: For some $\tilde{c}, \tilde{C}, A > 0$,
\beq\label{condnormsdual}
 \forall s \in \partial \DD, \forall z\in \C^n, \qquad \tilde c  |z|^q - A\le \phi^\ast_s( z) \le \tilde C  |z|^q + A.
\eeq
The typical example of boundary data verifying a $p$-uniform growth condition will be  powers of norms, $\phi(s, w) = \|w\|_s^p
$, where $\{\| \cdot \|_s\}_{s\in \pD}$ are a family of (uniformly equivalent) $\R$-norms on $\C^n$.

So we are given our boundary data $\phi:\pD \times \C\to \R \cup \{ + \infty \}$ with convex fibers
satisfying $p$-uniform growth conditions, or equivalently a family of convex function $\{\phi_s\}_{s\in \pD}$ on $\C^n$, and we would like to extend it to $\DD\times \C^n$, i.e., we would like to define $\phi_t$ for $t\in \DD$.

This is a Dirichlet problem, and from the point of view of complex analysis and HCMA, a natural procedure is to construct the largest possible PSH function which
does not exceed $\phi$ on the boundary, which is known as the
upper (or Perron-Bremermann) envelope. Let us first define, for $(t,z) \in \DD \times \CC^n$,
\begin{equation} \label{def_mu} \mu(t,z) = \mu(z) = \sup_{s \in \pD} \phi(s, z). \end{equation}
The reason for introducing $\mu$ is to impose some mild growth conditions on the PSH solution, as we are working on unbounded domains. Other choices would be possible; here we will just ask
that no interpolant of the family of functions $\{ \phi_s \}_{s \in \pD}$ should be greater than the supremum
of this family of functions (we can replace $\mu(z)$ by $C |z|^p+A$ when $\phi$ satisfies a $p$-uniform growth condition, but it is not natural for the definition below to depend on $p$). So we define for $t\in \DD$ and $w\in \C^n$,
\begin{multline*}
 \hat\phi(t, w) := \sup\{\psi(t,w) \; ;\\
   \psi :\DD \times \C^n \to \R \; \textrm{ is } \PSH, \ \psi \leq \mu \textrm{ on } \DD\times\CC^n \textrm{ and }  \psi \preceq \phi \textrm{ on } \pD \times \CC^n \}.
\end{multline*}
When we write that $\psi \preceq \phi$ on $\pD \times \CC^n$, we mean the following: For almost any $s \in \pD$ and for any $z \in \CC^n$,
$$ \limsup_{(0,1) \times \CC^n \ni (r, w) \rightarrow (1, z)} \psi(r s, w) \leq \phi(s, z). $$
It is known that this
upper envelope $\hat{\phi}$ may be interpreted as a weak  solution of the HCMA equation ~\eqref{MAnorm}, see for instance~\cite{BT,S}.

Another possibility, if one is rather guided by complex interpolation, is to introduce  for $t\in \DD$ and $w\in \C^n$,
\beq\label{def2}
\check \phi(t,w) := \inf\{ P[\phi(s, f(s))](t) \; :\ f \in H^p(\DD, \CC^n) \textrm{ with } f(t)=w\}.
\eeq
Note that  assumption~\eqref{condnorms} ensures that the function $\phi_s(f(s))$ is in $L^1(\pD)$, and hence its Poisson integral is well-defined.  For simplicity, we shall often adopt the following notation: for $p\in [1, +\infty]$,
$$ \vec{H}^p:=H^p(\DD, \C^n), \qquad \vec{L^p}:= L^p(\pD, \C^n)
\AND L^p := L^p (\pD, \CC). $$
As mentioned, the holomorphic functions we work with  have values in $\C^n$.

Yet, a third possibility is to have in mind the duality theorem for complex interpolation and the AWS theorem, and to introduce a polynomial-like hull.  For this hull we will use functions of the form
\beq\label{defL}
 L_{f_0,f}(t,w):= \Re (f(t) \cdot w)  -  P[\Re f_0](t), \qquad t\in  \DD,   \, w\in \C^n,
 \eeq
associated to $f \in \vec{H}^1$ and $f_0 \in L^1$.
The minus sign is there for cosmetic reasons, only (so it evokes the Legendre transform). The function  $L_{f_0,f}$ is harmonic in $t$ and linear in $w$, and by construction it is PSH on $\DD\times \C^n$. As is customary, we identify between a function in $\vec{H}^p$ and its boundary values in $L^p(\partial \DD, \CC^n)$.
Denoting as always $q=p/(p-1)$, we now introduce for $t\in \DD$ and $w\in \C^n$,
\begin{eqnarray*}
\tphi(t,w) &:=&\sup\big\{ L_{f_0,f}(t,w) \,;\  (f_0,f)\in L^1 \times \vec{H}^q,   \;
L_{f_0,f} \preceq  \phi  \ \textrm{on }  \pD \times \C^n \big\} \\
&=&\sup\big\{ L_{f_0,f}(t,w) \,;\  (f_0,f)\in L^1\times \vec{H}^q,   \;
\phi^\ast_s(f(s))\le  \Re f_0(s)  \ \textrm{for a.e. }  s \in \pD  \big\} \\
& = &\sup_{f\in \vec{H}^q}  \Re( f(t) \cdot w) -  P[\phi_s^\ast(f(s))](t),
\end{eqnarray*}
where we used the definition of the Legendre's transform of $\phi_s$ to pass from the first to the second line, and the monotonicity of the Poisson extension to get the third line. Note that the assumption on $\phi$, in the form~\eqref{condnormsdual}, ensures that the function $s\to \phi^\ast_s(f(s))$ belongs to $L^1$ when $f\in \vec{H}^q$.
Observe that by construction, the fibers $\tphi_t:=\tphi(t, \cdot)$ are convex on $\C^n$ for $t\in \DD$ fixed, thus $\tphi$ is a fiberwise convex function. Also by monotonicity it is clear
that $\tphi_t$ satisfies the $p$-uniform growth conditions, i.e., the inequalities in (\ref{condnorms}) hold for all $(s,z) \in \DD \times \CC^n$.

It follows easily  from the three definitions above (up to technical details that will be discussed later) that
\begin{equation}\label{easydirection}
\tphi\le \hat\phi \le \check\phi, \qquad \text{ on }\  \DD\times \C^n.
\end{equation}

The following theorem is, in our opinion, the key
to the understanding of our interpolation. It establishes the existence of a holomorphic disc, passing through a given point
in $\DD \times \CC^n$, such that $\tphi$ is harmonic along this disc.

\begin{mainthm}\label{mainthm}
Let $p\in (1, +\infty]$ and let $\{\phi_s\}_{s\in \pD}$ be a measurable family of convex functions
 satisfying $p$-uniform growth conditions~\eqref{condnorms}. Then for every $(t_0,w_0)\in \DD \times \C^n$ for which $\tphi(t_0, w_0) < \infty$
  there exists a holomorphic function $F\in \vec{H}^p=H^p(\DD, \C^n)$  with $F(t_0) = w_0$ such that
$$ \tphi(t, F(t)) = P \left[ \phi(s, F(s)) \right ](t) \qquad \qquad (t \in \DD). $$
In particular, $t \to \tphi(t, F(t))$ is harmonic in $\DD$.
\end{mainthm}

The condition that $\tphi(t_0, w_0) < \infty$ is relevant only in the case where $p = +\infty$,
as $\tphi$ is automatically finite when $p \in (0, +\infty)$.
We view Theorem \ref{mainthm}
as a functional version of the AWS theorem. In fact, the case $p = +\infty$ is
a generalization of the linear hull variant of the AWS theorem, as our boundary data is assumed fiberwise-convex and measurable,
but no global compactness assumption is made. The following proposition is an addendum to Theorem \ref{mainthm}.
We say that a convex function $\psi: \CC^n \rightarrow \RR \cup \{+ \infty \}$ is strictly convex if $\psi( (z+w) / 2 ) < (\psi(z) + \psi(w))/2$
for any distinct points $z,w \in \CC^n$ for which $\psi(z)$ and $\psi(w)$ are finite.

\begin{prop}[``foliation''] We work under the notation and assumptions of Theorem \ref{mainthm}.
Assume additionally that $\phi_s: \CC^n \rightarrow \RR \cup \{+ \infty \}$ is strictly-convex for all $s \in \pD$.
Then for any $(t_0, w_0) \in \DD \times \CC^n$, the holomorphic function $F = F_{t_0, w_0}$ from Theorem \ref{mainthm} is uniquely determined.
Moreover, suppose that $F_{t_0, w_0}(t) = F_{t_1, w_1}(t)$ for some $t, t_0, t_1 \in \DD$ and $w_0, w_1 \in \CC^n$. Then $$ F_{t_0, w_0} \equiv F_{t_1, w_1}. $$
\label{mainthm_more}
\end{prop}

We see from Proposition \ref{mainthm_more} that the collection of holomorphic discs
\begin{equation}\label{deffoliation}
 \cF(\phi) = \{ F_{t_0, w_0} \, ; \, (t_0, w_0) \in \DD \times \CC^n \, ; \, \tphi(t_0, w_0) < \infty \}
 \end{equation}
is a foliation of $\Omega := \{ (t_0, w_0) \in \DD \times \CC^n \, ; \, \tphi(t_0, w_0) < \infty \}$ in the following
sense: For any $(t,w) \in \Omega$ there exists a unique $F \in \cF(\phi)$ with $F(t) = w$.
Note that along each leaf of our foliation, the function $\tphi$ is harmonic, and not necessarily constant as in the case of $\C$-norms.
As alluded before, the existence of the foliation allows us to prove several crucial properties of interpolation, among which is the equality of hulls:

\begin{maincor}[``three hulls coincide''] \label{cor:eqhulls}
Let $p\in (1, +\infty]$ and $\phi: \pD \times \CC^n \rightarrow \RR \cup \{ + \infty \}$ be a fiberwise convex
function satisfying the $p$-uniform growth conditions. Then,
$$ \tphi = \hat\phi = \check\phi =: [\phi] \qquad \textrm{ in }\DD \times \C^n. $$ \label{prop1}
Moreover, when $p < \infty$ the function $[ \phi]$ is PSH in $\DD \times \CC^n$ with $[\phi] \preceq \phi$ on $\pD \times \CC^n$.
\end{maincor}

More precise information about the boundary values of $[\phi]$ will be provided in Proposition \ref{prop_311} below
in the case $1 < p < \infty$. 

%add DCE
So we have introduced a method of interpolation that we denoted by $[\phi]$ for a given suitable family a functions $\{\phi_s\}_{s\in \pD}$. An obvious property satisfied by this interpolation is reiteration. Let $\{\phi_s\}_{s\in \pD}$  measurable family of convex functions on $\C^n$ satisfying the
$p$-uniform growth conditions, and $\{[\phi]_t\}_{t\in \DD}$ the associated interpolated family. Given $r\in (0,1)$, then the  family of function $\{\psi_s\}_{s\in \pD}$ defined by
$$\psi_s := [\phi]_{rs}$$
will satisfy the $p$-growth condition, and its interpolation satisfies
$$[\psi]_t = [\phi]_{rt}, \qquad \forall t \in \DD.$$
To see this, we can note that from the definitions $\tilde{\phi}_{rt} \le \tilde{\psi}_t $ and $\check \psi_t \le \check\phi_{rt}$.

Another natural question related to interpolation would be to interpolate linear operators (between normed spaces). We will address this question  in the last section. 

We move now to one of our main goals, that is the following duality theorem, which in short says that $[\phi^\ast]=[\phi]^\ast$.

\begin{maincor}[``duality theorem''] \label{cor:duality}
Let $p\in (1, +\infty)$ and let $\{\phi_s\}_{s\in \pD}$  be a measurable family of convex functions on $\C^n$ satisfying
$p$-uniform growth conditions, and $\{[\phi]_t\}_{t\in \DD}$ the associated interpolated family, with the notation from Corollary~\ref{cor:eqhulls}.
Introduce the family $\psi_s := (\phi_s)^\ast$ for $s\in \partial \DD$, and let $\{[\psi]_t\}_{t\in \DD}$ be the associated interpolated family. Then, for any $t\in \DD$ we have
$$[\psi]_t = [\phi]_t^\ast \qquad \textrm{ on } \C^n.$$
\end{maincor}

We move on to discuss the boundary values of the interpolant. If $[\phi]$ is to be a honest interpolation
of $\phi$, it should tend to $\phi$ at the boundary. The following theorem provides a rather satisfactory
answer to this question, as the radial convergence of the fibers of $[\phi]$ to the boundary data $[\phi]$
is locally uniform in $\CC^n$. However, we were only able to cover the range $p \in (1, \infty)$,
as the proof of the following proposition relies upon the duality theorem.

\begin{prop}[``boundary values'']
Let $p\in (1, +\infty)$ and $\{\phi_s\}_{s\in \pD}$  be a measurable family of convex functions on $\C^n$ satisfying
$p$-uniform growth conditions.
Then for almost any $s \in \partial \DD$,
\begin{equation}
  [\phi]_{r s} \stackrel{r \rightarrow 1^-}{\longrightarrow} \phi_s \label{eq_1045} \end{equation}
locally uniformly in $\CC^n$.
\label{prop_311}
\end{prop}

Our results  show that $[\phi]$, which coincides with the the
upper envelope
solution of HCMA, is indeed a solution of the Dirichlet problem for the HCMA
equation in the following sense: The function $[\phi]$ is PSH, it attains the correct boundary values, and through any point
there passes a holomorphic disc along which $[\phi]$ is harmonic, i.e.,
we have a \MA foliation in the sense of Bedford--Kalka
\cite{Bedford-Kalka}. Moreover, this solution admits a Hopf--Lax type expression~\eqref{def2}.
Note that apart from the fiberwise convexity and the growth conditions, our only assumption on the boundary data is measurability, which a priori
might seem a rather weak assumption for the HCMA equation.

\medskip Finally, in the case where the boundary value is differentiable and strictly-convex in the $z$-variables (but only measurable in $t$),
we may assert in the next theorem that the foliation $\cF(\phi)$, defined in~\eqref{deffoliation}, associated with $\phi$ induces
the foliation $\cF(\phi^*)$ via the gradient map. Given a smooth function $\psi: \DD \times \CC^n \rightarrow \RR$ we denote
$$ \frac{\partial \psi}{\partial z} (t,z) = \left( \frac{\partial \psi}{\partial z_1},\ldots,\frac{\partial \psi}{\partial z_n} \right) \in \CC^n.
$$
%which is half the conjugate of the gradient map in $z$. 
Thus, $\frac{\partial \psi}{\partial z} $ is a function from $D \times \CC^n$ to $\CC^n$.

\begin{maincor}(``dual foliation'') \label{cor6}
Let $p\in (1, +\infty)$ and $\{\phi_s\}_{s\in \pD}$  be a measurable family of convex functions on $\C^n$ satisfying
$p$-uniform growth conditions. Assume that for any $s \in \pD$, the function $\phi_s: \CC^n \rightarrow \RR$
is differentiable and strictly-convex. Then for any $t \in \DD$, the function $[\phi]_t: \CC^n \rightarrow \RR$ is differentiable
and strictly-convex. Moreover, for any $F \in \cF(\phi)$, the function
$$ G(t) = 2\frac{\partial [\phi]}{\partial z} (t, F(t))  \qquad \qquad (t \in \DD)  $$
is holomorphic, belongs to $\cF(\phi^*)$, and
$$ [\phi](t, F(t)) + [\phi^*](t,G(t)) = \Re(F(t) \cdot G(t)) \qquad \qquad (t \in \DD). $$
\end{maincor}

%Before moving on to the proofs of all of these statements, let us comment on the relation between
%the complex interpolation method and a simpler real method. Assume that our boundary data depends
%only on the real parts of the coordinates, i.e., for $(s, z_1,\ldots,z_n) \in \pD \times \CC^n$,
%$$ \phi_s(z_1,\ldots, z_n) = \psi_s(\Re z_1,\ldots,\Re z_n). $$
%In this case, $4 \partial^2 \phi_s / (\partial z_j \overline{\partial z_k})$ equals to
%$\partial^2 \phi_s / (\partial x_j
%\partial x_k)$, where $z_j = x_j + \sqrt{-1} y_j$.
%Hence the equation (\ref{MAnorm}) reduces to a homogenous {\it real}  Monge--Amp\`ere (HRMA) equation. See Rauch and Taylor \cite{RT} for information
%on this equation. In particular, it is known that a weak solution for HRMA may be constructed
%as follows: For $(t, z) \in \DD \times \CC^n$,
%$$ [\phi](t,z) = \inf \left \{ \sum_{i=1}^N \lambda_i \phi_{s_i}(z_i) \, : \,
%N \geq 1, \lambda_1,\ldots,\lambda_N \geq 0, \sum_{i=1}^N \lambda_i = 1, \, \sum_{i=1}^N \lambda_i s_i = t,  \sum_{i=1}^N \lambda_i z_i = z \right \}. $$
%That is, the epigraph of $[\phi]$ is the convex hull of the epigraph of $\phi$. Thus, the complex interpolation method is a
%generalization of the real method, where holomorphic discs replace the r\^ole of line segments.

After this long introduction
and description of results, it is time for some refreshing proofs.

\section{Existence of foliations}
\label{sec_foli}

This section is devoted to the proof of Theorem~\ref{mainthm}.
Throughout this section, we fix $p\in (1, +\infty]$ and  $\{\phi_s=\phi(s, \cdot)\}_{s\in \pD}$  a measurable family of convex functions on $\C^n$
satisfying a $p$-uniform growth conditions~\eqref{condnorms}. Set $q = p / (p-1)$.

We begin with the following simple lemma:
\begin{lemma}  For any holomorphic function $F \in \vec{H}^p=H^p(\DD, \CC^n)$ and any $t \in \DD$,
$$ \tphi(t, F(t)) \leq P \left[ \phi(s, F(s)) \right ](t). $$ \label{lem0}
\end{lemma}

\begin{proof} Let $(h_0, h) \in L^1 \times \vec{H}^q$ be such that $L_{h_0,h} \preceq \phi$ on $\pD \times \CC^n$. The function
$$ \alpha(t) = L_{h_0,h}(t, F(t)) = \Re ( h(t) \cdot F(t) ) - P [\Re h_0](t) \qquad \qquad (t \in \DD) $$
is harmonic in $t \in \DD$. Since $h \in \vec{H}^q$ and $F \in \vec{H}^p$, the function $t \to h(t) \cdot F(t)$ belongs to the Hardy space $H^1(\DD, \CC)$.
We conclude that the function $\alpha(t)$ is in $P[L^1]$. Since $L_{h_0,h} \preceq \phi$ on $\pD \times \CC^n$,
for almost any $s \in \pD$,
$$ \alpha(s) = \lim_{r \rightarrow 1^-} \alpha(r s) \leq \limsup_{r \rightarrow 1^-} L_{h_0,h}(rs , F(rs)) \leq
\limsup_{(r, w) \rightarrow (1, F(s)) } L_{h_0,h}(rs , w) \leq \phi(s, F(s)), $$
where $r \in (0,1), w \in \CC^n$ and  we used the fact that $\lim_{r \rightarrow 1^-} F(rs) = F(s ) \in \CC^n $ for almost any $s \in \pD$. Since $\alpha \in P[L^1]$, for any $t \in \DD$,
$$ L_{h_0,h}(t, F(t)) = \alpha(t) = P[\alpha(s)](t) \leq P[ \phi(s, F(s)) ](t). $$
The lemma follows by the definition of $\tphi$ as the supremum over all such functions $L_{h_0, h}$.
\end{proof}

Our strategy for the proof of Theorem ~\ref{mainthm} is to find a holomorphic disc $F \in H^p(\DD, \CC^n)$, with $F(t_0) = w_0$, for which
the inequality in Lemma \ref{lem0} becomes an exact equality for at least one value of $t$, say for $t = 0$. Later on, this would  imply
that $\tphi(t, F(t))$ is harmonic, thanks to the following standard lemma:

\begin{lemma} Let $\alpha: \DD \rightarrow \RR \cup \{ -\infty \}$ be a non-positive function with $\alpha(0) = 0$.
Assume that $\alpha$ equals the supremum of a family of harmonic functions in $\DD$. Then $\alpha \equiv 0$.
\label{lem_1042}
\end{lemma}

\begin{proof} Let $\eps > 0$. Since $\alpha(0) = 0$, by the definition of the supremum there exists a
harmonic function $h$ on $\DD$ with $h(0) \geq -\eps$. Since $\alpha$ is non-positive, the harmonic function $h$
is non-positive as well. The Harnack inequality implies that for any $z \in \DD$,
$$ h(z) \geq -\frac{2 \eps}{1 - |z|}. $$
Therefore $\alpha(z) \geq -2 \eps / ( 1 - |z| )$. Since $\eps > 0$ was arbitrary, we see that $\alpha(z) \geq 0$
for all $z \in \DD$. The function $\alpha$ is assumed non-positive, and consequently it vanishes.
\end{proof}

%
%\begin{proof} Since $\alpha$ is the supremum of harmonic functions in $\DD$ which are continuous,
%then $\alpha$ is lower-semi continuous and hence measurable. Furthermore, it follows that
%\begin{equation}  \alpha(0) \leq \frac{1}{\pi} \int_{\DD} \alpha. \label{eq_1007} \end{equation}
%Since $\alpha(0) = 0$ with $\alpha$ non-positive, then (\ref{eq_1007})
%implies that $\alpha \equiv 0$ almost everywhere in $\DD$. Denote $A = \{ t \in \DD \, ; \, \alpha(t) = 0 \}$,
%a set of full measure in $\DD$. Let $t_0 \in A$ and $\eps > 0$.
%We claim that
%\begin{equation}
%\inf \left \{ \alpha(t) \, ; \, t \in B(t_0, (1 - |t_0|) / 2) \right \} \geq - \eps,
%\label{eq_1038} \end{equation}
%where $B(t_0, r) = \{ t \in \CC \, ; \, |t - t_0| < r \}$.
%Indeed, since $\alpha(t_0) =0$ then there exists a harmonic function $b \leq \alpha$ with $b(t_0) > -\eps/4$.
%Then $b$ is non-positive, and by Harnack's inequality, $b$ is at least $-\eps$ on the ball $B(t_0, (1 - |t_0|) / 2)$. Thus  (\ref{eq_1038}) is proven. Since (\ref{eq_1038})
%is valid for all $\eps > 0$, we conclude that  $\alpha \equiv 0$ in the set
%\begin{equation}  \bigcup_{t \in A} B \left(t, \frac{1-|t|}{2} \right). \label{eq_1040_} \end{equation}
%The set $A$ is of full measure, hence it is dense in $\DD$ and the union of the balls in (\ref{eq_1040_}) equals the entire disc $\DD$.
%Therefore $\alpha \equiv 0$ in $\DD$.
%\end{proof}

Recall that we are given a point $(t_0, w_0) \in \DD \times \CC^n$ with $\tphi(t_0, w_0) \in \RR$, and we need to find a holomorphic disc $F \in H^p(\DD, \CC^n)$, with $F(t_0) = w_0$,
for which the inequality of Lemma \ref{lem0} becomes an equality for at least one value of $t$.
Without loss of generality we may assume that $t_0 = 0 \in \DD$. Indeed, if $t_0 \neq 0$, then we may apply a fractional-linear (M\"obius) transformation
in the $t$-variable, and reduce matters to the case $ t_0 = 0$.
The central ingredient in the proof of Theorem~\ref{mainthm} is the following lemma,
whose proof boils down to an application of the Hahn--Banach theorem in an appropriate  space. Although it is possible to adapt and extend the approach of Alexander--Wermer~\cite{AW},  our proof will  be somehow closer  to Slodkowski's work~\cite{slodki} which relies  on Hardy spaces and on the F\&M Riesz Theorem.
%Set
%
%%Y: added:
%
%$$
%
%\vec{H}^p=H^p(\DD, \CC^n).
%
%$$

\begin{lemma} There exists a holomorphic function $F \in \vec{H}^p=H^p(\DD, \CC^n)$ with $F(0) = w_0$,
such that for any $g \in L^q( \partial \DD, \CC^n)$
$$ \int_{\partial \DD} \left[ \Re(g(s) \cdot F(s)) - \phi^\ast(s, g(s)) \right]  \, d\sigma(s) \le \tphi(0, w_0),$$
and as a consequence,
\begin{equation}
  \int_{\pD} \phi(s, F(s)) d \sigma(s)  \le \tphi(0, w_0). \label{eq:lem2}
  \end{equation}
\label{lem2}
\label{lem1}
\end{lemma}

With this lemma in hand, the proof of Theorem~\ref{mainthm} is immediate, the idea being that "subharmonicity" of $\tphi$ and the somewhat "opposite" inequality~\eqref{eq:lem2} force harmonicity. Let us postpone the proof of the lemma, and present the short details of this first.

\begin{proof}[Proof of Theorem~\ref{mainthm}]
Let $F \in \vec{H}^p$ be the holomorphic disc from Lemma \ref{lem1} which satisfies  $F(t_0) = w_0$. Denote
$$ \beta(t) = \tphi(t, F(t)) \qquad \qquad (t \in \DD). $$
Then $\beta(t)$ is the supremum of a family of harmonic functions, of the form $t \to L_{h_0, h}(t, F(t))$,
where $(h_0, h) \in L^1 \times \vec{H}^q$ and $L_{h_0, h} \preceq \phi$ on $\pD \times \CC^n$.
Consequently, also the function
$$ \alpha(t) = \tphi(t, F(t)) - P \left[ \phi(s, F(s)) \right ](t) \qquad \qquad (t \in \DD) $$
is a supremum of a family of harmonic functions. The function $\alpha$ is non-positive, according to Lemma \ref{lem0},
while $\alpha(0) \geq 0$ by Lemma \ref{lem2}. We may now invoke Lemma \ref{lem_1042} to conclude that $\alpha \equiv 0$,
completing the proof of the theorem.

\end{proof}

It remains to prove the central Lemma~\ref{lem1}.

\begin{proof}[Proof of Lemma~\ref{lem1}] Let us introduce the set
\begin{eqnarray*}
U &:=& \{  (g_0,g)\in  L^1\times \vec{L}^q\; : \;   P[\phi^\ast (s,g(s))](0)< P[ \Re g_0(s)](0)\} \\
&=&  \Big\{  (g_0,g)\in  L^1 \times \vec{L}^q\; : \; \int_{\partial \DD} \phi^\ast (s,g(s)) \, d\sigma(s)< \int_{\partial \DD} \Re g_0\, d\sigma\Big\}.
\end{eqnarray*}
We claim that this is an open, convex subset of $L^1\times \vec{L}^q$. Indeed, the convexity is clear. Moreover, the map
$$g\to \int_{\partial D} \phi^\ast (s,g(s)) \, d\sigma(s)$$
is convex and bounded in any ball of the normed space $\vec{L}^q$, thanks to the $p$-uniform growth conditions.
This ensures that this map is continuous (see, e.g.~\cite[Lemma 2.1]{ET}), and therefore  $U$ is open.
Recall the notation~\eqref{defL}, and introduce  the affine subspace
 $$E=\{ (f_0,f)\in L^1 \times \vec{H}^q \; :\  L_{f_0,f}(0,w_0) = \tphi(0,w_0)\}\subset   L^1 \times \vec{L}^q.  $$
We claim that the open  convex set $U$ and the subspace $E$ are disjoint.  Indeed, had $(f_0,f)$ belonged to their intersection, we would have
$$ \tphi(0,w_0) =  L_{f_0,f}(0,w_0)  =\Re( f(0)\cdot  w_0) -P[\Re f_0](0) <  \Re(f(0)\cdot  w_0)  -  P[\phi_s^\ast (f(s))](0) \le \tphi(0,w_0), $$
in contradiction. According to the Hahn-Banach theorem, $U$ and $E$ can be separated by a continuous linear functional
 $T \in (L^1 \times \vec{L}^q)^\star=L^\infty\times \vec{L}^p$:
\begin{equation}\label{sep}
 \forall (g_0,g) \in U , \  \Re T((g_0,g)) < \alpha, \AND  \forall (f_0,f)\in E, \ \Re T((f_0,f))=\alpha,
\end{equation}
for some $\alpha \in \R$. Let us apply (\ref{sep}) with $(g_0, g) = (k,0)$ for a certain fixed number $k \in \RR$, i.e., $g_0$ and $g$
are constant functions. Then for any $k \in \RR$,
$$ \Re T((k,0))< \alpha \ \textrm{ if } \ k > \int_{\partial \DD} \phi^*(s, 0) d \sigma(s) \AND
\Re T((k,0))= \alpha \ \textrm{ if } \ -k=\tphi(0,w_0) .$$
Since $T((k,0)) = k \cdot T((1,0))$, this shows that $$ \Re T( (1,0) ) < 0. $$
By multiplying $T$ and $\alpha$ by a positive
constant, we may thus assume that
\begin{equation} \Re T( (1,0) ) = -1 \qquad \text{and consequently} \qquad \alpha = \tphi(0,w_0). \label{eq_1352} \end{equation}
Our next goal is to establish the following, crucial,  representation:
\begin{equation}\label{eval}
\forall (f_0, f) \in L^1 \times \vec{H}^q,  \qquad \Re T((f_0,f)) = L_{f_0,f}(0,w_0)= \Re(f(0)\cdot w_0) - P[\Re f_0](0).
\end{equation}
Indeed,  for any $(f_0,f)\in L^1 \times \vec{H}^q$ we have that $(f_0 + L_{f_0,f}(0,w_0) - \tphi(0, w_0), f) \in E$. Hence from (\ref{sep}) and (\ref{eq_1352}),
\begin{align*} \Re T((f_0,f)) & = \Re T \left ( (f_0 + L_{f_0,f}(0,w_0) - \tphi(0, w_0), f) \right) + L_{f_0,f}(0,w_0) - \tphi(0,w_0) \\ & = \alpha+ L_{f_0,f}(0, w_0)- \tphi(0,w_0) = L_{f_0,f}(0, w_0).
\end{align*}
The representation (\ref{eval}) has several consequences. We need to recall first that the linear functional $T
\in (L^1 \times \vec{L}^q)^\star=L^\infty\times \vec{L}^p$ is given by a pair $(F_0, F) \in L^{\infty} \times \vec{L}^p$ so that
\begin{equation}   \forall (g_0,g) \in L^1\times \vec{L}^q, \quad  T((g_0,g))= \int_{\partial \DD} (g\cdot F)\, d\sigma - \int_{\partial \DD} g_0 \, F_0 \, d\sigma. \label{eq_1551}
\end{equation}
(We artificially put  a minus sign in front of $F_0$ for consistency with previous and later notation.)
This function $F$ ought to be the desired holomorphic disc. Note that for any $f_0\in L^1$,
$$-\int_{\partial \DD} \Re( f_0\, F_0) d\sigma= \Re T((f_0,0)) = L_{f_0,0}(0, w_0) = -P[\Re f_0](0) = -\int_{\partial D} \Re f_0\, d\sigma$$
This shows that $F_0\equiv 1$.
Next, note that for any holomorphic function $f\in H^q(D, \C^n)$ which satisfies $f(0)=0$ (it suffices to take  the functions $f(t)=t^m z_0$ with $m \geq 1$ and $z_0\in \C^n$)  we have, in view of~\eqref{eval}, that
$$ \int_{\pD} \Re(f\cdot F) \, d\sigma= T((0, f)) = L_{0,f}(0,w _0)= \Re (f(0) \cdot w_0) = 0. $$
By applying this also to the function $i f$, we conclude that
 $$ \int_{\pD} f\cdot F \, d\sigma= 0. $$
This shows that the Poisson integral $P[F]: \DD \rightarrow \CC^n$ is holomorphic, as all Fourier coefficients with negative indices of $F$ vanish (see e.g. \cite[Theorem 3.12]{katz}).
Thus $F \in H^p(\DD, \CC^n)=\vec{H}^p$.
For any fixed $z \in \CC^n$, apply (\ref{eval}) and (\ref{eq_1551}) with $f \in \vec{H}^p$ being the constant function $z$. This yields, by the mean-value property,
$$\Re(z\cdot F(0)) = \int_{\pD} \Re(z\cdot F)\, d\sigma = T \left( (0,z) \right) = L_{0,z}(0,w_0)= \Re(z \cdot w_0). $$
Therefore $F(0) = w_0$.
Finally, by (\ref{sep}) and (\ref{eq_1551}) we have, for any $(g_0, g) \in U$,
$$ \int_{\pD} \Re(g\cdot F)\, d\sigma - \int_{\pD}\Re g_0 \, d\sigma < \alpha = \tphi(0,w_0), $$
which implies that for any $g \in \vec{L}^q$,
\begin{equation}\label{lem:ineq1}
\int_{\partial \DD} \Big( \Re(g\cdot F)\, d\sigma - \phi^\ast(s, g(s)) \big) \, d\sigma \leq \tphi(0, w_0).
\end{equation}
This completes the proof of the first inequality of the lemma.

To deduce~\eqref{eq:lem2}, we would like to pick  an "almost" optimal choice
of the function $g$. This is an exercice in real analysis, and it can be carried out in the following way.  For $M > 0$  denote
$$ \phi_M(s, w) = \sup_{z \in \CC^n, |z| \leq M}  \big\{\Re(z\cdot w) - \phi^*(s, z)\big\}. $$
Since $\phi = (\phi^*)^*$, the function $\phi_M$ cannot exceed $\phi$, and for large $M$ the function $\phi_M$ is a good approximation to $\phi$.
More precisely,  $\phi_M \nearrow \phi$ as $M \rightarrow \infty$, pointwise in $\pD \times \CC^n$. In particular, for almost any $s \in \pD$,
 $$ \phi_M(s, F(s)) \nearrow \phi(s, F(s)) \qquad  \textrm{as} \ M \rightarrow \infty. $$ From the monotone convergence theorem, it suffices
to prove that for any $M, \eps > 0$,
$$ \int_{\pD} \phi_M(s, F(s)) d \sigma(s)  \le \tphi(0, w_0) + \eps. $$
Fix $M, \eps > 0$. It follows from the growth condition (\ref{condnormsdual}) that $\phi_M(s, w) < +\infty$ for all $(s,w) \in \pD \times \CC^n$.
Therefore, for almost any $s \in \pD$, there exists $g(s) \in \CC^n$ with $|g(s)| \leq M$ and
$$ \phi_M(s, F(s)) \leq \eps + \big\{ \Re(g(s) \cdot F(s)) - \phi^*(s, g(s))\big\}. $$
It is certainly possible to select $g$ in a measurable way. The function $g: \DD \rightarrow \CC^n$ is bounded, and in particular, $g \in L^q(\pD, \CC^n)$. It thus follows
from~\eqref{lem:ineq1} that
$$ \int_{\pD} \phi_M(s, F(s)) d \sigma(s)  \leq \eps + \int_{\pD} \left[ \Re(g(s) \cdot F(s)) - \phi^*(s, g(s)) \right] d \sigma(s) \leq \tphi(0, w_0) + \eps, $$
completing the proof of Lemma~\ref{lem1}.

\end{proof}

\section{Equality of hulls and duality}

In this section we deduce several consequences from Theorem \ref{mainthm} that were formulated in Section \ref{sec_main}.
 As before, $p \in (1, +\infty]$, and $\phi: \pD \times \CC^n \rightarrow \RR \cup \{ + \infty \}$ is a measurable, fiberwise convex function satisfying the $p$-uniform growth conditions. Set $q = p / (p-1)$. Recall the definitions of $\tphi, \hat{\phi}$ and $\check{\phi}$ from Section \ref{sec_main}.
Recall from (\ref{def_mu}) that the function $\mu$ is the supremum over $s \in \pD$ of the boundary data. In the next two lemmas, we establish the "easy" directions~\eqref{easydirection}.

\begin{lemma} $\displaystyle \tphi \leq \hat{\phi}$ throughout $\DD \times \CC^n$.
\label{lem4}
\end{lemma}

\begin{proof} Applying Lemma \ref{lem0} with the constant function $F(t) = z_0$, for some $z_0 \in \CC^n$, we see that
$$ \tphi(t, z_0) \leq P \left[ \phi(s, z_0) \right](t) \leq \mu(z_0) \qquad \qquad (t \in \DD, z_0 \in \CC^n). $$
Let $(h_0, h) \in L^1 \times \vec{H}^q$ be such that $L_{h_0,h} \preceq \phi$ on $\pD \times \CC^n$. The function
$$ L_{h_0,h}(t, w) = \Re( h(t) \cdot w) - P[\Re h_0](t) $$
is pluri-harmonic, and in particular it is PSH in $\DD \times \CC^n$.
By the definition
of $\tphi$ we know that $L_{h_0, h} \leq \tphi$. It is also clear that $L_{h_0, h} \leq \mu$.
Therefore $L_{h_0, h}$ is a competitor in the definition of $\hat{\phi}$, so
$L_{h_0, h} \leq \hat{\phi}$. Since $\tphi$ is the supremum over all such functions $L_{h_0, h}$, the lemma follows.
\end{proof}

\begin{lemma} $\displaystyle \hat{\phi} \leq \check{\phi}$ throughout $\DD \times \CC^n$.
\label{lem5}
\end{lemma}

\begin{proof}
Let $\psi: \DD \times \CC^n \rightarrow \RR$ be a PSH function such that $\psi \leq \mu$ on $\DD \times \CC^n$
and $\psi \preceq \phi$ on $\pD \times \CC^n$. Let $(t_0, w_0) \in \DD \times \CC^n$,
and let $f \in \vec{H}^p$ be any function with $f(t_0) = w_0$. It suffices to prove that
\begin{equation}
\psi(t_0, w_0) \leq P[\phi_s(f(s))](t_0) \qquad \qquad \qquad (t \in \DD). \label{eq_1453}
\end{equation}
Indeed, $\hat{\phi}(t_0, w_0)$ is the supremum over $\psi$ of  the left-hand side of (\ref{eq_1453}),
while $\check{\phi}(t_0, w_0)$ is the infimum over $f$ of the right-hand side. Denote $\alpha(t) = \psi(t, f(t))$,
a subharmonic function in $\DD$. Since $\psi \leq \mu$, by the $p$-uniform growth conditions, for some $A, C > 0$,
$$ \alpha(t) \leq \mu(f(t)) \leq A + C |f(t)|^p \qquad \qquad (t \in \DD). $$
Therefore, for all $s \in \pD$,
$$ \bar{\alpha}(s) := \sup_{0 < r < 1} \alpha(r s) \leq A + C \cdot \sup_{0 < r < 1} |f(rs)|^p = A + C \left| \bar{f} (s) \right|^p, $$
where $\bar{f}(s) = \sup_{0 < r < 1} |f(r s)| $. Since $f \in L^p(\pD, \CC)$ and $p > 1$, the Hardy-Littlewood maximal function inequality
(see \cite[Section III.2.4]{katz}) implies that $\bar{f} \in L^p(\pD, \CC)$ as well. In particular, the function $\bar{\alpha}$ is bounded
from above by an $L^1(\pD)$-function. Since $\psi \preceq \phi$ on $\pD \times \CC^n$,
for almost any $s \in \pD$,
$$ \limsup_{r \rightarrow 1^-} \alpha(r s) = \limsup_{r \rightarrow 1^-} \psi(rs , f(rs)) \leq
\limsup_{(0,1) \times \CC^n \ni (r, w) \rightarrow (1, f(s)) } \psi(rs , w) \leq \phi(s, f(s)), $$
where we used the fact that $\lim_{r \rightarrow 1^-} f(rs) = f(s ) \in \CC^n $ for almost any $s \in \pD$.
Since $\alpha$ is subharmonic,
\begin{align*} \alpha(t_0) & =\lim_{r \rightarrow 1^-} \int_{\pD} P(t_0,s) \alpha(r s) d \sigma(s) \stackrel{``Fatou"}{\leq}
\int_{\pD} P(t_0,s) \left[ \limsup_{r \rightarrow 1^-} \alpha(r s) \right] d \sigma(s) \\ & \leq
\int_{\pD} P(t_0,s) \phi(s, f(s)) d \sigma(s) = P[ \phi(s, f(s)) ](t_0),
\end{align*} where the use of Fatou's lemma is legitimate as $\bar{\alpha}(s) = \sup_{r \in (0,1)} \alpha(rs)$ has an integrable majorant.
Since $\psi(t_0, w_0) = \alpha(t_0)$, the desired inequality (\ref{eq_1453}) follows.
\end{proof}

Note that we have established the inequalities $\tphi \leq \hat{\phi} \leq \check{\phi}$ without appealing to Theorem \ref{mainthm}. We proceed by proving the uniqueness of the holomorphic disc in Theorem \ref{mainthm},
under strict convexity assumptions.

\begin{proof}[Proof of Proposition \ref{mainthm_more}] Assume by contradiction that $F_1, F_2 \in \vec{H}^p$
are two distinct holomorphic discs with $F_1(t_0) = F_2(t_0) = w_0$ and
$$  P \left[ \phi(s, F_1(s)) \right] (t_0) = \tphi(t_0, w_0) = P \left[ \phi(s, F_2(s)) \right] (t_0). $$
Then the set $\{ s \in \pD \, ; \, F_1(s) \neq F_2(s) \}$ is of positive measure, as $F_1$ and $F_2$ are distinct elements of $\vec{H}^p$.
Denote $G = (F_1 + F_2) / 2$. Then $G \in \vec{H}^p$ satisfies $G(t_0) = w_0$ and by strict-convexity,
\begin{align*} \check{\phi}(t_0, w_0) & \leq P \left[ \phi(s, G(s)) \right] (t_0) = P \left[ \phi \left(s, \frac{F_1(s) + F_2(s)}{2} \right) \right] (t_0) \\ & <
\frac{P \left[ \phi(s, F_1(s)) \right] (t_0) + P \left[ \phi(s, F_2(s)) \right] (t_0)}{2} = \tphi(t_0, w_0), \end{align*}
in contradiction to the inequality $\tphi \leq \check{\phi}$ which was proven in  Lemma \ref{lem4}
and Lemma \ref{lem5}. Hence for any $(t_0, w_0) \in \DD \times \CC^n$ there exists at most one holomorphic disc $F = F_{t_0, w_0} \in \vec{H}^p$ for which
$$ \tphi(t_0, w_0) = P \left[ \phi(s, F(s)) \right] (t_0). $$
It follows that if $F_{t_0, w_0}$ and $F_{t_1, w_1}$ coincide at a single point in $\DD$, they must be equal in the entire disc $\DD$.
\end{proof}

Next we employ Theorem~\ref{mainthm} and show that the three different interpolation schemes coincide.

\begin{lemma}
\label{equalitylemma}
 $\displaystyle [\phi] := \tphi = \hat{\phi} = \check{\phi}$ throughout $\DD \times \CC^n$. \label{lem_1308}
\end{lemma}

\begin{proof}
Lemma \ref{lem4} and Lemma \ref{lem5} show that $\tphi \leq \hat{\phi} \leq \check{\phi}$ in the entire domain $\DD \times \CC^n$.
Fix $(t_0, w_0) \in \DD \times \CC^n$
for which $\tphi(t_0, w_0) < \infty$. Let $F \in \vec{H}^p$ be the function from Theorem \ref{mainthm}
with $F(t_0) = w_0$. We can then use $F$ as a test function in the definition of $\check \phi$, hence,
$$ \check{\phi}(t_0, w_0) \leq P[\phi(s, F(s))](t_0) = \tphi(t_0, w_0). $$
This shows that $\check{\phi} \leq \tphi$ in the set $\{ (t_0, w_0) \, ; \, \tphi(t_0, w_0) < \infty \}$.
The inequality is trivial outside this set, and thus we conclude that $\check{\phi} \leq \tphi$ at all points of $\DD \times \CC^n$.
%In view of Lemma~\ref{lem4} and Lemma\ref{lem5}, the proof is complete.
\end{proof}

We move on to the proof that $[\phi] := \tphi = \hat{\phi} = \check{\phi}$  is PSH.
Indeed, $\tphi$ and $\hat{\phi}$ are defined as a supremum of PSH functions, but it is not apriori clear
that these functions are upper semi-continuous, as is  required in order to deserve the title ``a plurisubharmonic function''.
In order to prove that $[\phi]$ is PSH, we shall need the barrier function constructed (only for $p < \infty$) in the following:

\begin{lemma} \label{unnecessarylemma}
Assume that $p \in (1, +\infty)$. Then there exists a continuous, fiberwise-convex function $U: \DD \times \CC^n \rightarrow \RR$
such that $[\phi] \leq U$ in $\DD \times \CC^n$ and for almost any $s \in \pD$,
\begin{equation}  \lim_{r \rightarrow 1^-} U_{rs}  = \phi_s \label{eq_1249}
\end{equation}
locally uniformly in $\CC^n$, where $U_t: \CC^n \rightarrow \RR$ is defined as usual via $U_t(z) = U(t,z)$. \label{lem_1318}
\end{lemma}

\begin{proof} For any $z \in \CC^n$, the function $s \to \phi(s, z)$ is bounded on $\pD$,
according to the $p$-uniform growth conditions. For $(t, z) \in \DD \times \CC^n$ denote
$$ U(t,z) = U_t(z) = P \left[ \phi(s,z) \right] (t) = \int_{\pD} P(t, s) \phi(s, z) d \sigma(s) = U_{\phi}(t, z). $$
It is clear that $U_t: \CC^n \rightarrow \RR$ is a convex function, since it is the average of a family of convex functions.
The Poisson integral of a bounded function tends to the original function radially almost everywhere. Therefore, for almost any $s \in \pD$,
\begin{equation} \lim_{r \rightarrow 1^-} U(r s, z) = \phi(s,z). \label{eq_1118} \end{equation}
We now let $z \in \CC^n$ vary. For almost any $s \in \pD$, the relation (\ref{eq_1118}) holds true for almost any $z \in \CC^n$.
Standard convex analysis (see \cite[Theorem 10.8]{Rockafellar}) allows us to upgrade the a.e. convergence in $z \in \CC^n$ to a locally-uniform convergence in $\CC^n$,
completing the proof of (\ref{eq_1249}). Next we show that $U$ is continuous.  Indeed, if $(t_N, z_N) \rightarrow (t,z) \in \DD \times \CC^n$,
$$ U_{\phi}(t_N, z_N) = \int_{\pD} P(t_N, s) \phi(s, z_N) d \sigma(s) \stackrel{N \rightarrow \infty} \longrightarrow \int_{\pD} P(t, s) \phi(s, z) d \sigma(s) = U_{\phi}(t, z), $$
where the use of the bounded convergence theorem is legitimate thanks to the $p$-uniform growth conditions.
Finally, by applying Lemma \ref{lem0}
(and using Lemma \ref{equalitylemma})
 with $F(t) \equiv z$, for a fixed $z \in \CC^n$, we see that $[\phi] = \tphi \leq U$.
\end{proof}

\begin{proof}[Proof of Corollary~\ref{cor:eqhulls}]
In view of Lemma \ref{lem_1308}, it remains to prove that the function $\hat{\phi}$ is PSH with $\hat{\phi} \preceq \phi$ on $\pD \times \CC^n$. By the definition of $\hat{\phi}$
as a supremum of a family of PSH functions, we know that for any $w = (t, z) \in \DD \times \CC^n$, any $v  \in \DD \times \CC^n$ and a sufficiently small $\eps > 0$,
\begin{equation}  \hat{\phi}(w) \leq \int_{\pD} \hat{\phi}(w + \eps s v) d \sigma(s). \label{eq_1316}
\end{equation}
Denote \begin{equation} \psi(w_0) := \limsup_{w \rightarrow w_0} \hat{\phi}(w) \qquad \qquad (w_0 \in \DD \times \CC^n).
\label{eq-1104}
\end{equation}
Then $\psi: \DD \times \CC^n\to \R$ is upper semi-continuous
%YR
 (it is
the upper semi-continuous regularization of $\hat\phi$)
and $\hat{\phi} \leq \psi \leq \mu$ on $\DD \times \CC^n$. Moreover, thanks to the bound $\psi \leq \mu$ we may apply Fatou's lemma and replace
$\hat{\phi}$ by $\psi$ in the inequality (\ref{eq_1316}). Hence $\psi$ is a PSH function.

\medskip Let $U$ be the barrier function from Lemma \ref{lem_1318}. Then $U$ is continuous with $\hat{\phi} = [\phi] \leq U$. It  follows
from the definition (\ref{eq-1104}) that $\psi \leq U$ as well throughout $\DD \times \CC^n$. Lemma \ref{lem_1318} implies that $U \preceq \phi$ on the boundary $\pD \times \CC^n$,
and consequently also $\psi \preceq \phi$ on $\pD \times \CC^n$. We conclude that $\psi$ is a legitimate competitor in the definition
of $\hat{\phi}$. Hence $\psi \leq \hat{\phi}$ and finally we see that $\hat{\phi} = \psi$ is a PSH function, with $\hat{\phi} \preceq \phi$ on $\pD \times \CC^n$.
\end{proof}

\begin{remark}
%YR
{\rm
We could have also defined $\hat \phi$
as
\begin{multline*}
 \hat\phi(t, w) := \sup\{\psi(t,w) \; ;\\
   \psi :\DD \times \C^n \to \R \; \textrm{ is } \PSH, \ \psi \leq\hbox{\rm usc}\, \mu \textrm{ on } \DD\times\CC^n \textrm{ and }  
\hbox{\rm usc}\,\psi \le \phi \textrm{ on } \pD \times \CC^n \},
\end{multline*}
where {\rm usc} $f$ denotes the upper semi-continuous regularization
{\rm usc} $f(v)=\limsup_{w\ra v}f(w)$. Then Lemma \ref{unnecessarylemma}
and the proof of Corollary \ref{cor:eqhulls} would not be needed
if one is familiar with standard arguments in the pluripotential
literature (e.g., \cite[\S6]{BT},\cite[\S4]{Klimek}).
Indeed, it is standard that the Perron--Bremermann envelope
lies below the boundary data (by comparing with the harmonic
majorant) and upper semi-continuity follows from the above 
more ``friendly" definition $\hat\phi$. However, our original 
definition of $\hat\phi$ is somewhat easier to compare
with the other interpolants $\tilde\phi$
and $\check\phi$. 
}
\end{remark}

So far, the consequences of Theorem \ref{mainthm} proven in this section did not require the Legendre transform.
From now on, duality will play a major r\^ole. We continue  with a proof of the duality theorem:

\begin{proof}[Proof of Corollary~\ref{cor:duality}]
Note that the family $\psi_s= \phi_s^\ast$ satisfies the $q$-uniform growth conditions, and $q > 1$ since we assumed that $p \in (1, \infty)$.
We may therefore apply Theorem~\ref{mainthm} both to $\phi$ (with $p$)  and to $\psi$ (with $q$). Let us fix $t_0$ and $w_0, z_0 \in \C^n$, and let $F\in \vec{H}^p,
H \in \vec{H}^q$ be the holomorphic discs through $(t_0, w_0)$ for $[\phi]$ and through $(t_0, z_0)$ for $[\psi]$, respectively, given by Theorem \ref{mainthm}.
The  function
$$t\to [\phi](t, F(t)) + [\psi](t, H(t)) - \Re(F(t)\cdot H(t)) \qquad \qquad (t \in \DD) $$
is harmonic, and in $P[L^1]$. Its radial boundary values on $\pD$ equal
$$ \phi(s, F(s)) + \psi(s, H(s)) - \Re(F(s) \cdot H(s)) $$
which is $a.e.$ nonnegative, by the definition of the Legendre transform. We deduce that at $t_0$ we have
$$ [\phi](t_0, w_0) + [\psi](t_0, z_0) - \Re(w_0\cdot z_0) \ge 0.$$
Since this holds for every $z_0$ and $w_0$, we find that
$$[\psi] \ge [\phi]^\ast.$$
The converse inequality follows from the different definitions of the hull. If we use the definition of $\tilde \psi$, we have, for any $w\in \C^n$,
$$[\psi](t,z) = \sup_{h\in \vec{H}^p} \{\Re(z\cdot h(t)) - P[\phi(s,h(s))](t)\}. $$
From Lemma \ref{lem0} we know that for any $h\in \vec{H}^p$,
$$[\phi](t,h(t)) \le P[\phi(s,h(s))](t).$$
So we find,
$$ [\psi](t,z) \le \sup_{h\in \vec{H}^p} \{\Re(z\cdot h(t)) -[\phi](t,h(t))\} = [\phi]^\ast(t,z), $$
since $\{h(t)\; : \; h\in \vec{H}^p\}= \C^n$, as  can be seen by taking constant holomorphic functions.
\end{proof}

We move on to the proof of Proposition \ref{prop_311}. We will use an upper barrier for $[\phi]$, and a similar upper barrier for $[ \phi^* ]$.
By using the duality theorem (Corollary \ref{cor:duality}), we will deduce that $[\phi]$ attains the correct boundary values.

\begin{proof}[Proof of Proposition \ref{prop_311}]
Write $U_{\phi}$ for the barrier function for $\phi$ from Lemma \ref{lem_1318}.
Since $p \in (1, \infty)$, also $q = p / (p-1) \in (1, \infty)$ and $\phi^*$ satisfies the $q$-uniform volume growth
conditions. Lemma \ref{lem_1318} thus provides
another barrier function $U_{\phi^*}$ for the function $\phi^*$. Thus $[\phi] \leq U_{\phi}, [\phi^*] \leq U_{\phi^*}$
and for almost any $s \in \pD$, for any $z \in \CC^n$,
\begin{equation}  \lim_{r \rightarrow 1^-} U_{\phi}(rs, z)  = \phi(s, z), \qquad
\lim_{r \rightarrow 1^-} U_{\phi^*}(rs, z)  = \phi^*(s, z).
 \label{eq_1249_} \end{equation}
 Let us fix such a point $s \in \pD$. Let $r_N \nearrow 1$ be an arbitrary sequence.
Set $f_N = [\phi]_{r_N s}$ and $g_N = [\phi^*]_{r_N s}$.
Then for $N \geq 1$, the functions $f_N, g_N: \CC^n \rightarrow \RR$
are convex, and by Corollary \ref{cor:duality} we know that $g_N = f_N^*$.
By standard convex analysis (see Rockafellar \cite[Theorem 24.5]{Rockafellar}) if a subsequence $f_{N_j}$ converges
to a limit convex function $f: \CC \rightarrow \RR$ locally-uniformly in $\CC^n$, then necessarily $g_{N_j}$ tends
to $f^*$, again locally uniformly in $\CC^n$.

\medskip In order to find a convergent subsequence we appeal to Theorem 10.9 in Rockafellar \cite{Rockafellar}, and
conclude that there is a subsequence such that $f_{N_j}$ tends to a limit convex function $f: \CC^n \rightarrow \RR$, locally uniformly in $\CC^n$.
Consequently $g_{N_j}$ tends to $f^*$ locally-uniformly in $\CC^n$. Recall that $[\phi] \leq U_{\phi}$. It follows from (\ref{eq_1249_}) and from the choice of $s \in \pD$ that for any $z \in \CC^n$,
\begin{equation}  f(z) = \lim_{j \rightarrow \infty} f_{N_j}(z)
= \lim_{j \rightarrow \infty} [\phi]_{r_n s} \leq \limsup_{r \rightarrow 1^-} U_{\phi}( rs, z)  = \phi_s(z). \label{eq_1040}
\end{equation}
Similarly, we see that
\begin{equation}
f^*(z) = \lim_{j \rightarrow \infty} g_{N_j}(z)
= \lim_{j \rightarrow \infty} [\phi^*]_{r_n s} \leq \limsup_{r \rightarrow 1^-} U_{\phi^*}( rs, z)  = \phi^*_s(z). \label{eq_1110}
\end{equation}
The Legendre transformation reverses order, and hence from (\ref{eq_1040}) we see that $\phi_s^* \leq f^*$. But $f^* \leq \phi_s^*$ according to (\ref{eq_1110}).
Therefore $f^* = \phi_s^*$ and hence $f = \phi_s$.

 \medskip To conclude, we proved that for almost any $s \in \pD$ and for any sequence $r_N \nearrow 1$, the following holds: There exists a subsequence such that
 as $j \rightarrow \infty$, the sequence of functions $[\phi]_{r_{N_j} s}$ tends to $\phi_s$, locally uniformly in $\CC^n$. This proves (\ref{eq_1045}). Indeed,
 if (\ref{eq_1045}) fails, then there exists $z_0 \in \CC^n, \eps > 0$ and a sequence $r_N \nearrow 1$ such that the $L^{\infty}$-norm
 of the function  $[\phi]_{r_{N_j} s} - \phi_s$ is at least $\eps$ on the ball $B(z_0, \eps) \subseteq \CC^n$, contradicting the existence of the subsequence
 above.
\end{proof}

The proof of Proposition \ref{prop_311} may be adapted, in a straightforward manner,
to the case of non-tangential convergence to the boundary instead of a radial convergence to the boundary.
We omit the details. We move on to recall a few basic properties of the Legendre transformation.
Let $\psi: \CC^n \rightarrow \RR$ be a convex function and let $z_0 \in \CC^n$. According to \cite[Theorem 26.4]{Rockafellar} there exists $w_0 \in \CC^n$ such that
$$ \psi^*(w_0) + \psi(z_0) = \Re(z_0 \cdot w_0). $$
Moreoever, by \cite[Theorem 25.1]{Rockafellar}, this vector $w_0 \in \CC^n$ is unique if and only if $\psi$ is differentiable
at the point $z_0 \in \CC^n$, and in this case,
$$ w_0 = 2\frac{\partial \psi}{\partial z}(z_0). $$
Finally, recall (e.g., \cite[Theorem 26.3]{Rockafellar}) that $\psi$ is differentiable in all of $\CC^n$ if and only if $\psi^*$ is strictly-convex.

\begin{proof}[Proof of Corollary \ref{cor6}] Fix $F \in \cF(\phi)$.
Since $\phi_s$ is a differentiable function, then for almost all $s \in \pD$ there exists a unique point $T(s) \in \CC^n$ for which
\begin{equation}  \phi_s(F(s)) + \phi^*_s(T(s)) = \Re(F(s) \cdot T(s)). \label{eq_1543} \end{equation}
Moreover, $T(s) = 2\partial \phi_s / \partial z (F(s))$.
We now leave the boundary and enter the disc. Fix $t_0 \in \DD$ and set $z_0 = F(t_0)$. By the definition of $\tphi$, the function $\tphi_{t_0}: \CC^n \rightarrow \RR$
is the supremum over a family of linear functions on $\CC^n$, which satisfies $p$-uniform growth conditions.
We would like
to prove that the convex function $[\phi]_{t_0}: \CC^n \rightarrow \RR$ is differentiable at the point $z_0 \in \CC^n$. That is, that the vector
 $w_0 \in \CC^n$ such that
\begin{equation} [\phi]_{t_0}(z_0) + [\phi^*]_{t_0}(w_0) = \Re(z_0 \cdot w_0), \label{eq_1128}
\end{equation}
is uniquely determined. Let $w_0 \in \CC^n$ be such a vector, and let us prove that $w_0 = G(t_0)$, where $G \in
H^q(\DD, \CC^n)$ satisfies $G(s) = T(s)$ for almost all $s \in \pD$. It is clear that such $G$, if exists, is uniquely determined by $F$.
To this end, apply Theorem \ref{mainthm} with $\phi^*$,
and find $G \in \cF(\phi)$ with $G(t_0) = w_0$. That is, $G \in H^q(\DD, \CC^n)$ and
$$ [\phi^*](t, G(t)) = P\left[ \phi^*(s, G(s)) \right](t) \qquad \qquad (t \in \DD). $$
Note that the function
$$ [\phi](t,F(t)) + [\phi^*](t,G(t)) - \Re(F(t) \cdot G(t)) \qquad \qquad (t \in \DD) $$
is harmonic. This function is also non-negative, by the definition of the Legendre transform and the duality theorem. However, by (\ref{eq_1128})
this function vanishes at $t_0 \in \DD$. We conclude that this function vanishes identically, i.e., for all $t \in \DD$,
\begin{equation}  P\left[ \phi(s, F(s)) \right](t)  + P\left[ \phi^*(s, G(s)) \right](t)  = [\phi](t,F(t)) + [\phi^*](t,G(t)) = \Re(F(t) \cdot G(t)). \label{eq-1310}
\end{equation}
By considering the boundary values of the last equation, we learn that for almost all $s \in \pD$,
\begin{equation}  \phi_s(F(s)) + \phi^*_s(G(s)) = \Re(F(s) \cdot G(s)). \label{eq_1136} \end{equation}
From (\ref{eq_1543}) we learn that  $G(s) \equiv T(s)$ for almost any $s \in \pD$, as claimed. Consequently $[\phi]_{t_0}$ is differentiable at $z_0$ with
\begin{equation}  G(t_0) = 2\frac{\partial [\phi]}{\partial z}(t_0, z_0). \label{eq_1304}
\end{equation}
We now let $t_0 \in \DD$ vary. Note that the function $G = P[T]$ is determined by $F$ and $\phi$, and does not depend on the choice of $t_0 \in \DD$.
We may therefore repeat our analysis, and conclude that
$$ G(t) = 2\frac{\partial [\phi]}{\partial z}(t, F(t)) \qquad \qquad (t \in \DD). $$
In particular, this function is holomorphic. In view of (\ref{eq-1310}), we see that we proved all of the assertions
of the corollary, except for the strict-convexity of $[\phi]_{t_0}$. This follows by duality: By rerunning the argument for $\phi^*$,
we conclude that $[\phi^*]_{t_0}$ is differentiable in $\CC^n$, and hence $[\phi]_{t_0}$ is strictly-convex.
\end{proof}

%%%%%%%%%%%%%%%%%%%%%%%%%%%%%%%%%%%%%%%%%%%%%%%%%%%%%%%%%%%%%

\section{Complex interpolation of $\R$-norms}
\label{sec_last}

Let us finally get back to the question of doing complex interpolation of a family of (finite dimensional) real norm spaces, that is of interpolating a family $\{\|\cdot\|_s\}_{s\in \pD}$ of $\R$-norms on $\C^n$. We assume that this family is measurable in $s$, and also that the norms are uniformly equivalent, that is for all $s\in \pD$, $m |\cdot| \le \|\cdot\|_s \le M |\cdot|$, for some $m, M>0$.

We cannot take  directly $\phi(s,\cdot)=\|\cdot\|_s$ as boundary data in Theorem~\ref{mainthm} since these functions do not have superlinear growth as we require there. Nevertheless we can apply Theorem~\ref{mainthm} to
$$
\phi_p:=\phi_{p,s}(\cdot):= \|\cdot\|_s^p/p, \qquad s\in \pD,
$$
for $p>1$. The $1/p$ is there for cosmetic reason, because we anticipate the duality theorem (which holds for the Legendre's transform), but we can discard it since interpolation is linear: for a positive contant $r>0$,
$$ [r\phi]= r [\phi].$$
 %(Notice that here we have changed our previous definition of $\phi_o$ slightly by dividing by $p$.)
It is  immediate that for each $t$ in the disk $[\phi_p](t, \cdot)$ is also the $p$-th power of a norm (divided by $p$), so we can use this to define a {\it complex interpolation of $\R$-norms}.

To summarize,  given our family $\{\|\cdot\|_s\}_{s\in \pD}$ of $\R$-norms on $\C^n$, we can define  an interpolation method, $p$-interpolation,  for each $p$, where the family $\|\cdot\|_{p,t}$ of interpolated norms is given by
$$\|z\|_{p, t}:= \Big(\big[  \|\cdot\|_s^p\big](t, z) \Big)^{1/p}, \qquad t\in \DD, \ z \in \C^n.$$

Let us describe some properties of this $p$-interpolation, that follows from our results.
\begin{enumerate}
\item[i)]
For $t_0\in \DD$ and $w_0\in \C^n$, we can find $F\in \vec{H}^p=H^p(\DD, \C^n)$ such that $F(t_0)=w_0$ and $t \to  \|F(t)\|_{p,t}^p$ is harmonic (and in $P[L^1]$).
\item[ii)] The function $(t,z)\longto  \|z\|_{p, t}^p $ is PSH on $\DD \times \C^n$, with boundary limit $  \|z\|_{p, s}^p$ as $t\to s$ radially, for $a.e.$ $s\in \pD$.
\item[iii)] In terms of duality of norms~\eqref{dualnorm}, our duality theorem asserts that the dual of the norm $\|\cdot\|_{p,t}$ is equal the norm obtained by  the $q$-interpolation at $t$ of the family of dual norms $\{\|\cdot\|_s^\star\}_{s\in \pD}$, in short,
$$\big(\|\cdot\|_{p, t}\big)^\star = (\|\cdot\|^\star)_{q, t}$$
with $\frac1p + \frac1q =1$. Indeed, in the sense of the Legendre's transform duality,  for $t\in \DD$,
\begin{multline*}
\frac1q \Big(\big(\|\cdot\|_{p, t}\big)^\star\Big)^q
= \Big(\frac1p \big( \|\cdot\|_{p, t}\big)^p \Big)^\ast
= \big[ \big(\frac1p \|\cdot\|_s^p)\big)\big]^\ast(t,\cdot)
=  \big[ \big(\frac1p \|\cdot\|_s^p\big)^\ast\big](t,\cdot) =
\big[ \frac1q (\|\cdot\|_s^\star)^q\big](t,\cdot)
\end{multline*}
Note that in the case $p=2$, the definition of $2$-interpolation  is the one given in the Introduction in~\eqref{definterpolation2}.  Let us emphasize that this $2$-interpolation method is self-dual
%added DCE
and exact in the sense of interpolation.

%added DCE 
\item[iv)] If we are given a fixed $\R$-norm $\|\cdot\|$ and a positive measurable function $f$ on the circle that is bounded and bounded away from zero, then the $p$-interpolation of the family of norms $\{f(s)\|\cdot\|\}_{s\in \pD}$  at $t\in \DD$ is given by $\big(P[f^p](t)\big)^{1/p} \|\cdot\|$. 
\label{constants}
\end{enumerate}

% added DCE
It seems we have moved too quickly away from the most natural question : is this interpolation an interpolation method, in the sense that it allows to interpolate linear operators. Of course, an intriguing question here is the ability to interpolate between $\R$-linear operators, since we aim at using real norm spaces. The answer is yes if either the origin or the target space are interpolated (i.e., one normed space is fixed), as can be guessed by using one of the interpolation formulas. 

\begin{enumerate}
\item[v)] 
Let $n,d\ge 1$ and $Y_s=(\C^n, \|\cdot\|_{Y_s})$  be a family of $\R$-normed spaces  parameterised by $ s\in \pD$ (we assume as usual that the norms are measurable  and uniformly equivalent in the parameter $s\in \pD$) and $X=(\R^d, \|\cdot\|)$ a fixed real normed space. Assume we are given an family of $\R$-linear operators $A_t:\R^d \to \C^n$, $t\in \overline\DD$, 
with the property that the map $t\to A_t$ is holomorphic on  $\DD$ (in the sense that for fixed $x\in \R^d$, the map $t\to A_t x \in \C^d$ is holomorphic) and continuous on $ \overline\DD$.

Given $p>1$, we perform the $p$-interpolation of the norms $ (\|\cdot\|_s:=\|\cdot\|_{Y_s})_{ s\in \pD}$ and denote accordingly, for $t\in \DD$, $Y_t = (\C^n, \|\cdot\|_{p,t})$
the corresponding interpolated $\R$-normed spaces.  If we have
$$\forall s \in \pD, \qquad \|A_s\|_{ X \to Y_s}\le 1$$
then
$$\forall t \in \DD, \qquad \|A_t \|_{X\to Y_t}\le 1.$$
This follows from monotonicity of the interpolated norms (using for instance any of the possible definitions) : if we fix  $x_0\in \R^d$ with $\|x_0\|\le 1$, then the property $\|A_s x_0\|_s \le 1$ for all $s\in \pD$ extends to $t\in \DD$ since $t\to \|A_t x_0 \|_{p,t}$ is subharmonic. 

Analogously, we can interpolate the origin spaces. This follows by duality, by considering the (real) adjoints $A^\star$ of the operators $A$, which will share the same operator norms, and our duality theorem. The result is as follows. Let $n,d\ge 1$ and $X_s=(\C^n, \|\cdot\|_{X_s})$  be a family  of $\R$-normed spaces  parameterised by $ s\in \pD$  and $Y=(\R^d, \|\cdot\|)$ a fixed real normed space. Assume we are given an family of $\R$-linear operators $A_t:\C^n \to \R^d$, $t\in \overline\DD$, with the property that the map $t\to A_t^\star:\R^d\to \C^n$ is holomorphic on  $\DD$ and continuous on $ \overline\DD$.  Given $p>1$, we perform the $p$-interpolation of the norms $ (\|\cdot\|_s:=\|\cdot\|_{X_s})_{ s\in \pD}$ and denote accordingly, for $t\in \DD$, $X_t = (\C^n, \|\cdot\|_{p,t})$
the corresponding interpolated $\R$-normed spaces.  If we have
$$\forall s \in \pD, \qquad \|A_s\|_{X_s\to \R^d}\le 1$$
then
$$\forall t \in \DD, \qquad \|A_t \|_{X_t \to \R^d}\le 1.$$
 
 Finally, we can ask what happens if we want to interpolate both the origin and the target spaces for a family of operators $A_t: \C^n \to \C^d$. Let us give an alternative approach to the question. Given two $\R$-normed spaces $X=(\C^n, \|\cdot\|_X)$ and $Y=(\C^d, \|\cdot \|_Y)$ and a linear map $A:X \to Y$, the property that 
$$\|A\|_{X\to Y}\le 1$$
 is equivalent to the property that
$$\forall w \in \C^n, \forall x \in \C^d, \qquad \Re (Aw \cdot x) \le \frac1p \|w\|_X^p + \frac1q \|x\|_{Y^\star}^q$$
where $\frac1p + \frac1q = 1$. 
Assume we are given   $X_s=(\C^n, \|\cdot\|_{X_s})$ and  $Y_s = (\C^d, \|\cdot\|_{Y_s})$  two families  of $\R$-normed spaces  parameterised by $ s\in \pD$. 
For each family we assume that the norms are measurable  and uniformly equivalent in the parameter $s\in \pD$. Assume we are also given a family of $\R$-linear operators $A_t:\C^n \to \C^d$, $t\in \overline\DD$ such that the map $t\to A_t$ is holomorphic on  $\DD$ and continuous on $ \overline\DD$.
Given $p>1$, we perform the $p$-interpolation of the norms $ \{\|\cdot\|_{X_s}\}_{ s\in \pD}$ and  $\{\|\cdot\|_{Y_s}\}_{ s\in \pD}$ and denote accordingly, for $t\in \DD$
$$X_t = (\C^n, \|\cdot\|_{p,X_t}) \AND Y_t = (\C^d, \|\cdot\|_{p,Y_t}).$$
the corresponding $p$-interpolated $\R$-normed spaces. We ask when the bounds
$$\forall s \in \pD, \qquad \|A_s\|_{X_s\to Y_s}\le 1$$
ensure that
$$\forall t \in \DD, \qquad \|A_t \|_{X_t \to Y_t}\le 1?$$
We introduce the $q$-interpolation of the spaces $Y_s^\star = (\C^d, \|\cdot\|_{Y_s}^\star)$,  that we denote by $F_t$ for $t\in \DD$. Our duality theorem recalled above says exactly that
$F_t = Y_t ^\star$.
Now let us fix $t_0\in \DD$ and  $(w,x)\in \C^n\times \C^d$. Let $f\in  \vec{H}^p=H^p(\DD, \C^n)$ and $g\in  \vec{H}^q=H^q(\DD, \C^d)$ be the associated foliations: $f(t_0)= w$, $g(t_0)= x$ and the functions $t\to \|f(t)\|_{p,X_t}^p$ and $t\to \|g(t)\|^q_{F_t}=\|g(t)\|_{q,Y_t^\star}^q$ are  harmonic and in $P[L^1]$. If the function
$$\alpha(t) := \Re A_t f(t) \cdot g(t) - \frac1p \|f(t)\|_{p,X_t}^p - \frac1q \|g(t)\|_{q,Y_t^\star}^q$$
is  harmonic on $\DD$ (and in  $P[L^1]$), then we have a positive answer to the question: if it is nonnegative on the boundary, it will remain nonnegative at $t_0$. However, it the operators $A_t$ are only $\R$ linear, then we cannot conclude in general since the function $$(t,w,x) \to  \Re A_t w \cdot x $$
is no-longer guaranteed to pluri-harmonic, even if $A_t$ does not depend on $t$. In the situations discussed above, when one of the spaces is fixed, it follows that one of the foliations can be taken to be constant, and then we are fine. However, if we assume that the operators $A_t$ are $\C$-linear, then we are also fine, as the function above is indeed pluri-harmonic. So the most general interpolation theorems holds in the case of $\C$-linear operators between our $p$-interpolated $\R$-normed spaces.

Finally, let us mention that that we have discussed the case where the operator norms are bounded by one, but we can extend the result to arbitrary bounds. The precise bound for the operator norms between the interpolated spaces is given by the remark $iv)$ p.\pageref{constants}.

\end{enumerate}

\medskip

%modified DCE
To end this section, let us discuss the difference between $\C$-norms and $\R$-norms, and as a consequence show that our results contain the classical results on complex interpolation between (finite) dimensional complex normed spaces.

If we assume that our boundary
values are $\C$-homogenous in $z$, which is the case when we work with $\C$-norms,  the situation described
above is a lot more rigid than in the general case.  While the following proposition
is well-known in the classical theory of interpolation, which  coincides, as we shall see, with our
interpolation method, we decided to include its proof as it provides a different angle. It reveals why harmonic functions along leaves become constant.

\begin{prop} Let $p \in (1, +\infty)$ and let $\phi: \pD \times \CC^n \rightarrow \RR^+$
be a measurable, fiberwise-convex function satisfying $p$-uniform growth conditions.
Assume that the boundary values $\phi(t,z)$ are homogeneous of degree $p>0$ in $z$, i.e.,
$
\phi(t,\lambda z)=|\lambda|^p\phi(t,z)
$
for $\lambda$ in $\C$. Then:
\begin{enumerate}
\item[(i)] $[\phi]$ has the same property.
\item[(ii)] $\log [\phi]$ is plurisubharmonic.
\item[(iii)]
If $h$ is a holomorphic function of one variable such that $t\to [\phi](t, h(t))$
is harmonic on $\DD$, then  $t\to [\phi](t, h(t))$ is constant on $\DD$.
\end{enumerate}
\end{prop}
\begin{proof} The first claim is obvious from the definition.
In order to prove (ii), we use the following classical fact:  a function $\Phi$ on $\C^N$ (below $N=n+1$) is log-plurisubharmonic if there exists $k>0$ such that $|\lambda|^k\Phi(w)$ is plurisubharmonic as a function of $(\lambda, w)\in (\C\setminus\{0\})\times \C^N$.
By definition of plurisubharmonicity, it is enough to check this property when $N=1$, and by approximation, when $\Phi$ is smooth and positive. Then, the complex Hessian of $|\lambda|^k\Phi(z)$ with respect to $(\lambda, z)\in \C^2$
is
$$ \begin{pmatrix}
(\frac{k}2)^2| \lambda|^{k-2}\Phi(z) &\frac{k}2| \lambda|^{k-2}\b\l \partial_z \Phi \\
\frac{k}2| \lambda|^{k-2}\l \partial_{\b z} \Phi (z)&|\l|^k \partial^2_{z\b z}\Phi(z)
\end{pmatrix}, $$
and the determinant of this nonnegative matrix is equal to
$$
\Big(\frac{k}2\Big)^2| \lambda|^{2k-2} \big(\Phi(z)   \partial^2_{z\b z}\Phi(z) -\partial_z \Phi (z)\partial_{\b z} \Phi (z)\big) = 4 \Big(\frac{k}2\Big)^2| \lambda|^{2k-2}  \Phi(z)^2  \Delta \log\Phi(z).
$$

Now, by (i),
$$
|\lambda|^p \cdot [ \phi](t,z)=[\phi](t,\lambda z).
$$
This shows that $ (\lambda, t,z) \to |\lambda|^p \cdot [ \phi](t,z)$ is PSH, and
so, $\log [ \phi]$ is plurisubharmonic by the fact mentioned
earlier, and (ii) is proven. To deduce (iii), observe that if  $[\phi](t, h(t))$ is harmonic it must be constant
since its logarithm is subharmonic by (ii).
\end{proof}

So assume that we are given, as before, a family of norms $\{\|\cdot\|_s\}_{s\in \pD}$, but this time each $\|\cdot\|_s$ is a $\C$-norm on $\C^n$. Let us fix some $p>1$ and define, as above, the $p$-interpolated norms $\|\cdot\|_{p,t}$ at $t\in \DD$. It follows from our main theorem and from the previous proposition that for $t_0\in \DD$ and $w_0\in \C^n$ exists an holomorphic function $F$ with $F(t_0)=w_0$  such that $t\to \|F(t)\|_t^p$ is constant (and in $P[L^1]$), and therefore  we have
$$\|F(t)\|_{p,t}=\|w_0\|_{p, t_0},\qquad  \forall t \in \DD, \text{ and for a.e. } t\in \pD,$$
that is, we reproduce the result~\eqref{trivialfoliation} mentioned in the Introduction. In particular, our function $F$ which was in $\vec{H}^p$ is in fact in $\vec{H}^\infty$. Consequently, we can replace the mean with respect to the harmonic measure by the supremum (for a constant function, it is the same) in the definition of $\check \phi$, and obtain that
$$\|w_0\|_{p,t_0} = \inf\{ {\rm ess} \sup_{s\in \pD} \| F(s)\|_s \; : F \in H^\infty(\DD, \C^n) , \; F(t_0)=w_0 \} =\|w_0\|_{t_0}$$
where $\|\cdot\|_t$ refers to the classical complex interpolation~\eqref{definterpolation}. Therefore, in the case of interpolation of $\C$-norms, all the $p$-interpolation methods coincide with the usual interpolation. In particular, our results (foliations, equality of hulls, duality) apply and allow to reproduce  classical results on complex interpolation. More importantly, we hope this comparison between the situation of $\R$-norms and $\C$-norms sheds new light on complex interpolation itself.

%%%%%%%%%%%%%%%%%%%%%%%%%%%%%%%%%%%%%%%%%%%%%%%%%%%%%%%%%%%%%
\section*{Acknowledgments}

This work is based on the SQuaREs project award
``Interactions between convex geometry and complex geometry"
from the American Institute of Mathematics (AIM) and NSF.
The authors are grateful to AIM and its staff
for the funding, hospitality, and excellent working conditions
over the years 2011--2013.
This research was
supported by grants from ANR, BSF (2012236),
ERC (305629), NSF (DMS-0802923,1206284,1515703),
VR, and a Sloan Research Fellowship.
DCE wishes to thank Gilles Pisier for illuminating discussions on complex interpolation.
YAR is grateful to R.J. Berman and Chalmers Tekniska H\"ogskola
for their hospitality and support in Spring 2014 when part of this work was carried out.

\def\listing#1#2#3{{\sc #1}:\ {\it #2}, \ #3.}

\bigskip

{\sc Chalmers University of Technology and G\"oteborg University}

{\tt bob@chalmers.se}

\bigskip

{\sc Sorbonne Universit\'e, Institut de Mat\'ematiques de Jussieu}

{\tt dario.cordero@imj-prg.fr}

\bigskip

{\sc Weizmann Institute of Science and Tel-Aviv University}

{\tt boaz.klartag@weizmann.ac.il}

\bigskip

{\sc University of Maryland}

{\tt yanir@umd.edu}

\end{document}